\documentclass[11pt, a4paper, reqno]{amsart}
\usepackage[dvips]{epsfig}
\usepackage{amscd}
\usepackage{amssymb}
\usepackage{amsthm}
\usepackage{amsmath}
\usepackage{graphicx,enumitem}
\usepackage{latexsym}
\usepackage{cite}
\usepackage{multirow}
\usepackage{tabularx}
\usepackage{mathrsfs}
\usepackage{qtree}
\usepackage[linguistics]{forest}
\usepackage{tikz}

\usepackage[colorlinks]{hyperref}
\usepackage{color}
\usepackage{xcolor}
\usepackage{subcaption}
\usepackage[utf8]{inputenc}
\usepackage[T1]{fontenc}
\theoremstyle{plain}
\newtheorem{theorem}{Theorem}[section]
\newtheorem{lemma}[theorem]{Lemma}

\newtheorem{proposition}[theorem]{Proposition}
\newtheorem{corollary}[theorem]{Corollary}

\newtheorem{remark}[theorem]{Remark}

\allowdisplaybreaks

\setlist{leftmargin=15pt,labelindent=15pt}
\setlist[enumerate]{wide=0pt, leftmargin=15pt, labelwidth=15pt, align=left}

\title{ISOTROPY GROUP OF LOTKA-VOLTERRA DERIVATIONS }

\author[ Himanshu Rewri]{Himanshu Rewri}
\author[ Surjeet Kour]{Surjeet Kour}
\address[] {\newline
Department of Mathematics,
Indian Institute of Technology Delhi,
Hauz Khas, New Delhi, 110016, India.}
\email[] {himanshurewri50@gmail.com, koursurjeet@gmail.com}

\begin{document}

\begin{abstract}
In this paper, we study the isotropy group of Lotka-Volterra derivations of \newline $K[x_{1},\cdots,x_{n}]$, i.e., a derivation $d$ of the form $d(x_{i})=x_{i}(x_{i-1}-C_{i}x_{i+1})$. 
If $n=3$ or $n \geq 5$, we have shown that the isotropy group of $d$ is finite. However, for $n=4$, it is observed that the isotropy group of $d$ need not be finite. Indeed, for $C_{i}=-1$, we observed an infinite collection of automorphisms in the isotropy group of $d$.  Moreover, for $n \geq 3, ~~\text{and}~~C_{i}=1$, we have shown that the isotropy group of $d$ is isomorphic to the dihedral group of order $2n$. 
\end{abstract}
\maketitle

\section{Introduction}

Let $K$ denote a field of characteristic zero containing $\mathbb{Q}$, $K^* = K \setminus \{0\}$ and $A$ be a commutative $K$-algebra.  A $K$-linear map $d: A \rightarrow A$ is said to be a $K$-derivation if it satisfies the Leibnitz rule; $d(ab)=ad(b) + bd(a)$ for all $a, b \in A$. The set of all $K$-derivations is denoted by $\text{Der}_K(A)$. Let $\text{Aut}_K(A)$ be the set of all $K$-automorphisms of $A$.  Then $\text{Aut}_K(A)$ acts on $\text{Der}_K(A)$ by conjugation  as 
$$\rho \cdot d = \rho d \rho^{-1},  ~~~~~~ \text{for} ~~ \rho  \in  \text{Aut}_K(A) ~~~\text{and} ~~~ d \in \text{Der}_K(A).$$ 
For a fixed derivation $d \in \text{Der}_K(A)$, the stabilizer of the above action is a subgroup of $\text{Aut}_K(A)$ given by, $\{ \rho  \in \text{Aut}_K(A)~|~\rho d = d \rho \}$, and called the isotropy group of  $d$. In other words, the isotropy group of $d$ consists of all automorphisms of $A$ that commute with $d$ and is denoted by $\text{Aut}(A)_d$. It is easy to see that the isotropy groups are well-behaved with the conjugation; more precisely, if $d_1$ and $d_2$ are two derivations conjugate to each other by an automorphism $\sigma$, then $\text{Aut(A)}_{d_1} =\sigma (\text{Aut(A)}_{d_2})  \sigma^{-1}$. 
Further, if $A$ is a polynomial algebra over $K$, a $K$-derivation $d$ on $A$ corresponds to a polynomial vector field over the affine space $\mathbb{A}_K^n$ and in this case, elements of $\text{Aut}(A)_d$ correspond to the polynomial symmetries of that vector field.

In recent years, much work has been done on the isotropy group of simple derivations of the polynomial algebra (see, for instance, \cite{Balt}, \cite{BaltazarPan2021}, \cite{Bert}, \cite{Pan}, \cite{Yan}). A derivation $d$ is said to be simple if it has no non-trivial $d$ invariant ideal. For examples of simple derivations, one can see \cite{Baltazar2015}, \cite{Leq}, \cite{Kour2014}, \cite{Yan2019simple}. In \cite{Balt}, R. Baltazar conjectured that if $d$ is a simple derivation of  $K[x_{1},x_{2}]$, then the isotropy group of $d$ is finite.
This conjecture has been established affirmatively by L.G. Mendes and I. Pan, see \cite{Pan}. They have shown that the isotropy group of simple derivation is indeed trivial. The same assertion does not apply to derivations that are not simple. Indeed, the isotropy group of a non-simple derivation can be infinite. However, the isotropy group is trivial even in higher dimensions for a particular class of derivations.
In \cite{Bert}, L. N. Bertoncello and D. Levcovitz showed that if $d$ is a simple Shamsuddin derivation of $K[x_{1},\ldots,x_{n}]$, then the isotropy group of $d$ is trivial. Subsequently, D. Yan in 2024 proved that a Shamsuddin derivation of $K[x_{1},\ldots,x_{n}]$ is simple if and only if its isotropy group is trivial, see \cite{Yan}. Recently, other papers (see, for instance, \cite{Baltazar2021}, \cite{Dasgupta}) have aimed to study the isotropy groups of non-simple derivations over some special kinds of domains. In \cite{Baltazar2021}, R. Baltazar and M. Veloso describe the isotropy group of locally nilpotent derivations on Danielewski surfaces formed by the polynomial of the form $f(x_1)x_2-g(x_3)$ on $K[x_1, x_2, x_3]$. These results are then extended to generalized Danielewski surfaces and Danielewski varieties (which are further generalizations of Danielewski surfaces) by N. Dasgupta and A. Lahiri, see \cite{Dasgupta}. In \cite{Rewri}, the authors studied non-simple derivations of the form $x_{1}^{m}x_{2}^{u}\delta$, where $m,u\geq 0$ and $\delta$ is a simple derivation of $K[x_{1},x_{2}]$. It is shown that the isotropy group of such derivations is isomorphic to the product of two cyclic groups. 
 
This paper aims to study the isotropy group of Lotka-Volterra derivations of $K[x_{1},\ldots,x_{n}]$. Throughout the article, let $n \geq 3$ be an integer and $\mathbb{Z}_n$ be the additive group of integers modulo $n$. We list the elements of  $\mathbb{Z}_n$ as $\{1, 2, \ldots, n\}$. Further, the indexes of the variables $x_{1}, \ldots, x_{n}$ in the polynomial algebra  $K[X]=K[x_{1},\ldots,x_{n}]$ are regarded as elements of $\mathbb{Z}_n$. A $K$-derivation $d: K[X] \rightarrow K[X]$ of the form 
\begin{equation}\label{LVD}
  d(x_{i})=x_{i}(x_{i-1}-C_{i}x_{i+1})  
\end{equation}
where $C_{i} \in K$ for all $i \in \mathbb{Z}_n$, is called a Lotka-Volterra derivation. Lotka-Volterra derivations are examples of non-simple derivations, and they have many applications in various branches of science, such as plasma theory, biology, and physics. These occur as a system of differential equations for the predator and prey population as functions of time (see, for instance, \cite{Bogo}, \cite{Chau}). For more on the Lotka-Volterrra derivations, we refer the reader to \cite{Moulin1},\cite{Ziel}.\par
In section $2$, we establish some general results on the $K$-automorphisms of $K[X]$ that commute with the Lotka-Volterra derivations of $K[X]$, along with the study of the isotropy group of Lotka-Volterra derivations of $K[x_1, x_2, x_3]$. We prove that the isotropy group of Lotka-Volterra derivations of $K[x_1, x_2, x_3]$ are finite. For $d$, as defined in equation \eqref{LVD}, depending on the various values of $C_i$'s, we have listed elements of the isotropy group in Table \ref{Table1} and Table \ref{Table2}. If $1-C_{1}C_{2}C_{3} \neq 0$, then the isotropy group of $d$ is given in Table \ref{Table1}. Further, if $1-C_{1}C_{2}C_{3} = 0$, the isotropy group of $d$ is as mentioned in Table \ref{Table2}. In addition, if $C_1=C_2=C_3=1$, the isotropy group of $d$ is isomorphic to the dihedral group $D_{6}$ of order $6$.
 
In section $3$, we study the isotropy group of Lotka-Volterra derivations of $K[x_1, x_2, x_3, x_4]$. In this case, we first observed that the isotropy group of $d$  need not be finite. If $C_{i}=-1$ for $1 \leq i \leq 4$, we obtained an infinite family of automorphisms of $K[x_1, x_2, x_3, x_4]$ in the isotropy group.  Indeed, some non-linear automorphisms were observed in the isotropy group of $d$ (see Theorem \ref{infinite}). Further, we describe the isotropy group of $d$ for all values of $C_{i}$'s. It is shown that if $C_{i}$'s are not equal to $-1$, the isotropy group of $d$ is finite.

In section $4$, we studied the isotropy group of Lotka-Volterra derivations of the polynomial algebra $K[x_1, \ldots, x_n]$, where $n\geq 5$. We first observed that the isotropy group of Lotka-Volterra derivation is always finite. Further, we obtained some interesting results for particular values of $ C_i$'s. For example, we observed that if  $C_{i}=1$ for all $i$, then the isotropy group of Lotka-Volterra derivation is isomorphic to the dihedral group $D_{2n}$ of order $2n$, and if $C_{i}=0$ for all $i$, then the isotropy group is isomorphic to $\mathbb{Z}_n$. In addition, if $n$ is even and $C_i$'s are of alternate sign, the isotropy group is isomorphic to
$D_{2m}$, where $n=2m$.
 
\section{\textbf{ Isotropy group of Lotka-Volterra derivation of \texorpdfstring{$K[x_1, x_2, x_3]$}{} }}

 The main aim of this section is to study the isotropy group of Lotka-Volterra derivations of $K[x_{1},x_{2},x_{3}]$. Before we do so, we establish some general results on the $K$-automorphisms of $K[X]$ that commute with the Lotka-Volterra derivation of $K[X]$  defined in equation \eqref{LVD}. Let $\mathbb{N}_{0}=\mathbb{N}\cup \{0\} $ and for $\alpha = \left ( \alpha_{1}, \ldots,\alpha_{n} \right ) \in \mathbb{N}_{0} ^{n}$, let $X^{\alpha}$ denote the monomial $x_1^{\alpha_{1}}\cdots x_n^{\alpha_{n}}$. Further, as mentioned in the introduction, the indexes of the variables $x_1, \ldots, x_n$ in the polynomial algebra $K[X]$ are elements of $\mathbb{Z}_n= \{1, 2, \ldots, n\}$. More precisely, if $i$ and $j$ are integers such that $i \equiv j(\text{mod}~ n)$, then $x_i=x_j$. For $p \in \mathbb{N}_0$, let $A^{(p)}$ denote the set of all homogeneous polynomials of degree $p$ in $K[X]$. Then $A^{(p)}$ is a $K$-submodule of $K[X]$ and 
 
 $$K[X]=\bigoplus_{p \in \mathbb{N}_{0}} A^{(p)}.$$ 
 
Further, a polynomial $f \in K[X]$ has the decomposition $f=f_{0}+\cdots+f_{p}$, where $f_{i} \in A^{(i)} ~~\forall i \geq 0$. A $K$-derivation $\delta$ of $K[X]$ is called a homogeneous of degree $s$ if $ \delta \left ( x_{i} \right ) \in A^{(s+1)}$ for all $i \in \mathbb{Z}_n$. Observe that the Lotka-Volterra derivation defined in equation \eqref{LVD} is homogeneous of degree $1$. 
\begin{lemma}\label{totaldegree}
Let $d$ be the Lotka-Volterra derivation as defined in equation \eqref{LVD} and $\rho \in \text{Aut}(K[X])_{d}$. Then, for all $i \in\mathbb{Z}_{n}$, $\rho(x_{i-1}-C_{i}x_{i+1})$ is a polynomial in $K[X]$ of total degree $1$. 
\end{lemma}
\begin{proof}
For $i \in \mathbb{Z}_n$, we have
\begin{align}
   d(\rho(x_{i}^{2}))=\rho(d(x_{i}^{2})) 
   \Rightarrow  d ( \rho\left ( x_{i} \right )^{2} ) =\rho(2x_{i}d(x_{i}))
   =2\rho(x_{i})^2\rho(x_{i-1}-C_{i}x_{i+1}). 
\end{align}\label{N1}
Let deg$(\rho(x_{i}))=m_{i}$ and deg$(\rho(x_{i-1}-C_{i}x_{i+1}))=n_{i}$. Then, deg$ (d(\rho\left ( x_{i} \right )^{2}))\leq 2m_{i}+1 $ (as $d$ is homogeneous of degree 1), and by comparing the total degree in the equation \eqref{N1}, we get $2m_{i} + n_{i} \leq 2m_{i} + 1$. Hence, $n_{i} \leq 1$. Since $\rho$ is an automorphism, we have $n_{i}=1$.
\end{proof}

\begin{corollary}\label{Cor1}
Let $d$ be the Lotka-Volterra derivation as defined in equation \eqref{LVD}. If there exists  $k \in \mathbb{Z}_{n}$ such that deg$(\rho(x_{k-1}))=$ deg$(\rho(x_{k}))=1$, then deg$(\rho(x_{i}))=1$, for all $ i \in \mathbb{Z}_{n}$. 
 \end{corollary}
 \begin{proof}
Using Lemma \ref{totaldegree} recursively and the  given condition, we have
$$ \textup{deg}(\rho(x_{k-1-2t}))=\textup{deg}(\rho(x_{k-1}))=1 ~~~\text{and}~~~ \textup{deg}(\rho(x_{k-2t}))= \textup{deg}(\rho(x_{k}))=1~~ \text{for all}~~ t \geq 0.$$ 
Now, for $ 1 \leq i \leq n$, 
define $t_i=(k-i)/2$ if $(k-i)$ is even else $t_i=(k-i-1)/2$. Then,
deg$(\rho(x_{i}))=$ deg$(\rho(x_{k-2t_i}))=$ deg$(\rho(x_{k}))=1~ \text{if}~ t_i=(k-i)/2,~\text{and} 
    ~$deg$(\rho(x_{i}))=$ deg$(\rho(x_{k-1-2t_i}))=$ deg$(\rho(x_{k-1}))=1~ \text{if}~ t_i=(k-i-1)/2$.
 \end{proof}

\begin{corollary}\label{Cor2}
Let $d$ be the Lotka-Volterra derivation as defined in equation \eqref{LVD}. If $n$ is odd and there exists  $ k \in \mathbb{Z}_{n}$ such that $C_{k}=0$. Then deg$(\rho(x_{i}))=1$, for all $ i \in \mathbb{Z}_{n}$. 
 \end{corollary}
\begin{proof}
Let $ n= 2m+1$ for some $m \in \mathbb{N}$. As $C_k=0$, by Lemma \ref{totaldegree}, $\rho(x_{k-1})$ is a polynomial of degree 1. Further, repeated use of Lemma \ref{totaldegree} gives 
deg$(\rho(x_{k}))=$ deg$(\rho(x_{k-2m}))=$ deg$(\rho(x_{k-1-(2m+1)}))=$ deg$(\rho(x_{k-1}))= 1$. Now, by using  Corollary \ref{Cor1}, we get deg$(\rho(x_{i}))=1$ for all $ i \in \mathbb{Z}_{n}$.  
\end{proof}

The following lemma facilitates the computation of the 
  $\rho(x_{i})$'s degree  when $C_{i+1} \in K^* $.

\begin{lemma}\label{Importantlemma} 
Let $d$ be the  Lotka-Volterra derivation as defined in equation \eqref{LVD} and $\rho \in \text{Aut}(K[X])_{d}$. If $n\neq 4$ and there exists $ t \in \mathbb{Z}_{n}$ such that $C_{t+1} \in K^{*}$, then $\rho(x_{t})$ is a polynomial in $K[X]$ of total degree 1.
\end{lemma}
\begin{proof}
Let $\rho(x_{t})=f_{s}+f_{s-1}+\cdots+ f_{0}$, where $f_{i}$'s are homogeneous polynomials of degree $i$, $f_{s} \neq 0$ and $s\geq 1$. 
From Lemma \ref{totaldegree}, we know  $\rho(x_{t}-C_{t+1}x_{t+2})$ is a polynomial of total degree $1$.
 Let $\rho(x_{t}-C_{t+1}x_{t+2})=g_{0}+g_{1}$. Then, $$C_{t+1}\rho(x_{t+2}) = \rho(x_t)- g_{0}-g_{1}= f_{s}+f_{s-1}+\cdots+ f_{0} - g_{0}-g_{1}.$$ As $\rho \in \text{Aut}(K[X])_{d}$, we have,  $d\rho(x_{t}-C_{t+1}x_{t+2})=\rho d(x_{t}-C_{t+1}x_{t+2})$. \\
 Therefore, 
\begin{align}
d \rho (x_{t}-C_{t+1}x_{t+2}) & = \rho (x_{t})\rho(x_{t-1}-C_{t}x_{t+1})-C_{t+1}\rho(x_{t+2})\rho(x_{t+1}-C_{t+2}x_{t+3})\notag \\ ~& = (f_{s}+f_{s-1}+\cdots+ f_{0})\rho(x_{t-1}-C_{t}x_{t+1})  \notag  \\ & \hspace{.5cm}-(f_{s}+f_{s-1}+\cdots+ f_{0} - g_{0}-g_{1})\rho(x_{t+1}-C_{t+2}x_{t+3}). \label{N2}
\end{align}
Let $\rho(x_{t-1}-C_{t}x_{t+1})=h_{0}+h_{1}$ and $\rho(x_{t+1}-C_{t+2}x_{t+3})={p_{0}}+{p_{1}}$ where $h_{0},p_{0} \in K$, and $h_{1}, p_{1}$ are homogeneous polynomials of degree $1$. Since $\rho(x_{t}-C_{t+1}x_{t+2})$ is linear, $d\rho(x_{t}-C_{t+1}x_{t+2})$ is a polynomial of total degree at most $2$. 

If $s\geq 2$, then from equation \eqref{N2}, we have $f_{s}h_{1}-f_{s}p_{1}=0$. This implies, $h_{1}=p_{1}$ and $\rho(x_{t-1}-C_{t}x_{t+1}) - \rho(x_{t+1}-C_{t+2}x_{t+3})= h_{0}-p_{0} \in K$. Since $n \neq 4$, $x_{t-1} \neq x_{t+3}$ and $(x_{t-1}-(1+C_{t})x_{t+1} + C_{t+2}x_{t+3})$ is a non constant polynomial in $K[X]$. 
Therefore, $\rho(x_{t-1}-(1+C_{t})x_{t+1}+ C_{t+2}x_{t+3}) \notin K$. Hence, if $s \geq 2$, we have a contradiction.   
\end{proof}

\begin{remark}\label{remarkfor4variable}
If $n=4$, then from equation \eqref{N2} in Lemma \ref{Importantlemma}, $\rho((1+C_{i+2)}) x_{i-1}-(1+C_{i})x_{i+1})=h_{0}-p_{0} \in K$. This is possible only when  $C_{i}=C_{i+2}=-1$. So, for any $ i \in \mathbb{Z}_{4}$, if either $C_{i}$ or $C_{i+2}$ is not equal to $-1$, then $\rho(x_{i})$ is a polynomial of degree 1. 
\end{remark}

\begin{corollary}\label{reducedtotaldegree}
Let $d$ be the Lotka-Volterra derivation as defined in equation \eqref{LVD} and $\rho \in \text{Aut}(K[X])_{d}$. If $n \neq 4$, then $\rho(x_{i})$ is a polynomial of degree 1 for all $ i \in \mathbb{Z}_{n}$.
\end{corollary}
\begin{proof}
If $C_{i} \in K^{*}$ for all $ i \in \mathbb{Z}_{n}$, then by Lemma \ref{Importantlemma}, $\rho(x_{i})$'s are polynomials  of total degree 1. Now, if there exist an $i$ such that $C_{i}=0$, then by Lemma \ref{totaldegree}, deg$(\rho(x_{i-1}))=1$. Now, we have the following two possibilities:
\begin{enumerate}
    \item $C_{i+1}=0$. In this case, deg$(\rho(x_{i}))=1$, and the result follows from  Corollary \ref{Cor1}.
    \item $C_{i+1} \neq 0$. In this case, using Lemma \ref{Importantlemma}, we get deg$(\rho(x_{i}))=1$, and now the result follows from  Corollary \ref{Cor1}.
\end{enumerate}
\end{proof}

Now, we study the isotropy group of Lotka-Volterra derivations on $K[x_{1},x_{2},x_{3}]$. The following tables display the $K$-automorphisms of $K[x_{1},x_{2},x_{3}]$, that are in the isotropy group of Lotka-Volterra derivation for different values of $C_{1}, C_{2}$, and $C_{3}$.

\begin{table}[ht]
    \centering
      \begin{tabular}{ |c|c|} 
\hline
\textbf{$C_{1}, C_{2},C_{3}$} & $\mathbf{(\rho(x_{1}),\rho(x_{2}),\rho(x_{3}))}$ \\
\hline
 $C_{1},C_{2},C_{3}$ are arbitrary & $(x_{1},x_{2},x_{3})$\\
  \hline
$C_{2}C_{3}=1, C_{1}=-1 $ &  $(-C_{3}^{-1}x_{3},x_{2},-C_{3}x_{1})$\\
 \hline
$C_{1}C_{2}=1, C_{3}=-1 $ &  $(x_{1},-C_{2}x_{3},-C_{2}^{-1}x_{2})$\\
 \hline
 $ C_{1}C_{3}=1,C_{2}=-1 $ &  $(-C_{3}^{-1}x_{2},-C_{3}x_{1},x_{3})$\\
  \hline
\multirow{2}{*} {$C_{1}=C_{2}=C_{3} $} &  $(x_{2},x_{3},x_{1})$ \\  &  $(x_{3},x_{1},x_{2})$\\
\hline

\end{tabular}
    \caption{$\mathbf{1-C_{1}C_{2}C_{3} \neq 0}$}
    \label{Table1}
\end{table}

\begin{table}[ht]
    \centering
     \begin{tabular}{ |c|c|} 
\hline
\textbf{$C_{1}, C_{2}, C_{3}$} & $\mathbf{(\rho(x_{1}),\rho(x_{2}),\rho(x_{3}))}$\\
\hline
 $C_{1},C_{2}, C_{3}$ are arbitrary & $(x_{1},x_{2},x_{3})$\\
  \hline
$C_{2}C_{3}=1, C_{1}=1 $ &  $(-C_{3}^{-1}x_{3},-x_{2},-C_{3}x_{1})$\\
 \hline
$C_{1}C_{2}=1, C_{3}=1 $ &  $(-x_{1},-C_{2}x_{3},-C_{2}^{-1}x_{2})$\\
 \hline
 $C_{1}C_{3}=1, C_{2}=1 $ &  $(-C_{3}^{-1}x_{2},-C_{3}x_{1},-x_{3})$\\
  \hline
\multirow{2}{*} {$C_{1}=C_{2}=C_{3} $} &  $(x_{2},x_{3},x_{1})$\\ &  $(x_{3},x_{1},x_{2})$\\
\hline
\multirow{2}{*} {$C_{2}=C_{3}=-1,C_{1}=1$} & $(x_{1}+x_{2},-x_{2},-x_{2}+x_{3})$ \\ & $(-x_{2}+x_{3},x_{2},x_{1}+x_{2})$ \\
\hline
\multirow{2}{*} {$C_{1}=C_{2}=-1, C_{3}=1 $} &  $(x_{1},x_{1}+x_{3},-x_{1}+x_{2})$ \\&  $(-x_{1},-x_{1}+x_{2},x_{1}+x_{3})$ \\
\hline
\multirow{2}{*} {$C_{1}=C_{3}=-1,C_{2}=1$} &  $(x_{1}-x_{3},x_{2}+x_{3},-x_{3})$ \\
& $(x_{2}+x_{3},x_{1}-x_{3},x_{3})$ \\
\hline
\end{tabular}
    \caption{$\mathbf{1-C_{1}C_{2}C_{3} = 0}$}
    \label{Table2}
\end{table}

\begin{theorem}\label{Main1}
Let $d=x_{1}(x_{3}-C_{1}x_{2})\frac{\partial }{\partial x_{1}}+x_{2}(x_{1}-C_{2}x_{3})\frac{\partial }{\partial x_{2}}+x_{3}(x_{2}-C_{3}x_{1})\frac{\partial }{\partial x_{3}}$, be a derivation of $K[x_{1},x_{2},x_{3}]$. If $1-C_{1}C_{2}C_{3}\neq 0$, then the isotropy group of $d$ is as listed in Table \ref{Table1}.
\end{theorem}
\begin{proof}
Let $\rho \in  Aut( K[ x_{1},x_{2},x_{3} ] )_{d}$. By Corollary \ref{reducedtotaldegree}, we know that $\rho(x_{i})$'s are polynomials in $K[x_{1},x_{2},x_{3}]$ of total degree $1$.\\ For $ i \in \mathbb{Z}_{3}$, let 
$$\rho(x_{i})=c_{i}+\sum_{j=1}^3 c_{ij}x_{j},$$
where $c_{i}, c_{ij} \in K$, for all $ i, j \in \mathbb{Z}_{3}$.\newline
Then equation $d\rho\left ( x_{1} \right ) =\rho d\left ( x_{1} \right )$ gives,
\begin{equation}
\begin{aligned}
\sum_{j=1}^3 c_{1j}x_{j}(x_{j-1}-C_{j}x_{j+1})=
(c_{1}+ \sum_{j=1}^3 c_{1j}x_{j}) 
&\big((c_{3}+\sum_{j=1}^3 c_{3j}x_{j})-C_{1}(c_{2}+\sum_{j=1}^3 c_{2j}x_{j})\big).
\end{aligned}\label{Eq2}
\end{equation}
Similarly, $d\rho\left ( x_{2} \right ) =\rho d\left ( x_{2} \right )$ and $d\rho\left ( x_{3} \right ) =\rho d\left ( x_{3} \right )$ gives,
\begin{equation} 
\begin{aligned}
\sum_{j=1}^3 c_{2j}x_{j}(x_{j-1}-C_{j}x_{j+1})=(c_{2}+\sum_{j=1}^3 c_{2j}x_{j})
&\big((c_{1}+\sum_{j=1}^3 c_{1j}x_{j})- ~ C_{2}(c_{3}+\sum_{j=1}^3 c_{3j}x_{j})\big) \end{aligned}\label{Eq3}
\end{equation}
 and
 \begin{equation}
\begin{aligned}
\sum_{j=1}^3 c_{3j}x_{j}(x_{j-1}-C_{j}x_{j+1})=(c_{3}+\sum_{j=1}^3 c_{3j}x_{j}) 
&\big((c_{2}+\sum_{j=1}^3 c_{2j}x_{j})-C_{3}(c_{1}+\sum_{j=1}^3 c_{1j}x_{j})\big). \end{aligned}\label{Eq4}
\end{equation}

By comparing the coefficient of $x_{1}^{2},~x_{2}^{2}$ and $x_{3}^{2}$ in equations \eqref{Eq2}-\eqref{Eq4}, we have
{\begin{equation}\label{T1}
\left.\begin{matrix}
 x_{1}^{2} & x_{2}^{2} & x_{3}^{2}\\ 
c_{i1}(c_{(i-1)1}-C_{i}c_{(i+1)1})=0 & c_{i2}(c_{(i-1)2}-C_{i}c_{(i+1)2})=0 &  c_{i3}(c_{(i-1)3}-C_{i}c_{(i+1)3})=0
\end{matrix}\right\},
\end{equation}}
where $i \in  \mathbb{Z}_3$. \\
Since $\rho$ is an automorphism,  $$\rho(x_{i-1}-C_{i}x_{i+1})=(c_{i-1}-C_{i}c_{i+1})+\sum_{j=1}^{3} (c_{(i-1)j}-C_{i}c_{(i+1)j})x_{j}\notin K, ~~~~\text{for all} ~~ i \in \mathbb{Z}_3.$$ In particular, for $i =k$, 
$\rho(x_{k-1}-C_{k}x_{k+1})=(c_{k-1}-C_{k}c_{k+1})+\sum_{j=1}^{3} (c_{(k-1)j}-C_{k}c_{(k+1)j})x_{j} \notin K$. Hence, there exists a $ t \in \mathbb{Z}_{3}$ such that $c_{(k-1)t}-C_{k}c_{(k+1)t} \neq 0$, which by using equation \eqref{T1} implies  $c_{kt}=0$.

Now, by using the coefficient of $x_{t}^{2}$ in equation \eqref{T1} for $i=k+2$, we have $c_{(k+1)t}c_{(k+2)t}=0$. This implies, either $c_{(k+1)t}=0$ or $c_{(k+2)t}=0$ but not both (as $\rho$ is an automorphism). 
Without loss of generality, we assume that $c_{(k-1)t}= c_{(k+2)t}=0$. Then $c_{kt}-C_{k+1}c_{(k+2)t}=0$, and  
\begin{small}
$$\rho(x_{k}-C_{k+1}x_{k+2})= (c_{k}-C_{k+1}c_{k+2})+(c_{k(t-1)}-C_{k+1}c_{(k+2)(t-1)})x_{t-1}+(c_{k(t+1)}-C_{k+1}c_{(k+2)(t+1)})x_{t+1}.$$
\end{small}
Since $\rho(x_{k}-C_{k+1}x_{k+2}) \notin K$, either $c_{k(t-1)}-C_{k+1}c_{(k+2)(t-1)} \neq 0$ or $c_{k(t+1)}-C_{k+1}c_{(k+2)(t+1)} \neq 0$.

\textbf{Case 1.} If $c_{k(t+1)}-C_{k+1}c_{(k+2)(t+1)} \neq 0$. Then, from equation \eqref{T1} for $i=k+1$, $c_{(k+1)(t+1)}=0$.
Now, from the coefficient of $x_{t+1}^{2}$ in equation \eqref{T1} for $i=k$, we have $c_{k(t+1)}c_{(k-1)(t+1)}=0$. Thus, either $c_{k(t+1)}=0$ or $c_{(k-1)(t+1)}=0$, but not both. 

Let us consider first $c_{k(t+1)}=0$. In this case, $\rho$ is given by $\rho(x_{k-1})=c_{k-1}+c_{(k-1)(t-1)}x_{t-1}+c_{(k-1)(t+1)}x_{t+1}$, $\rho(x_{k})=c_{k}+c_{k(t-1)}x_{t-1}$, $\rho(x_{k+1})=c_{k+1}+c_{(k+1)(t-1)}x_{t-1}+c_{(k+1)t}x_{t} $.

Now, from the coefficients of $x_{t-1}^{2}$ in equation \eqref{T1}, we have 
\begin{equation}\label{T3}
 c_{p(t-1)}(c_{(p-1)(t-1)}-C_{p}c_{(p+1)(t-1)})=0  \hspace{1cm} \forall~~ p \in \{k,k\pm1 \}. 
\end{equation}
 Since $1-C_{1}C_{2}C_{3}\neq 0$, one can observe that, $c_{(p-1)(t-1)}-C_{p}c_{(p+1)(t-1)} \neq 0 $ for at least one $p$; otherwise, $(1-C_{1}C_{2}C_{3})c_{p(t-1)}=0$ for all $p  \in \{k,k\pm1 \}$, which is not possible as $\rho(x_{k}) \notin K$. Also,  $c_{(k-1)(t-1)}-C_{k}c_{(k+1)(t-1)}=0$ as $c_{k(t-1)} \neq 0$. Hence, either $c_{k(t-1)}-C_{k+1}c_{(k-1)(t-1)} \neq 0$ or $c_{(k+1)(t-1)}-C_{k-1}c_{k(t-1)} \neq 0 $. If $c_{k(t-1)}-C_{k+1}c_{(k-1)(t-1)} \neq 0$ then equation \eqref{T3}, for $p=k+1$, gives $c_{(k+1)(t-1)}=0$. Now, from the coefficient of $x_{t-1}^{2}$ in equation \eqref{T1} for $i=k$, the relation $c_{(k-1)(t-1)}c_{k(t-1)}=0$ gives $c_{(k-1)(t-1)}=0$. 
 Similarly, when $c_{(k+1)(t-1)}-C_{k-1}c_{k(t-1)} \neq 0 $, one can observe that $c_{(k-1)(t-1)}=c_{(k+1)(t-1)}=0$. Therefore, 
\begin{align}
\rho(x_{k-1})&=c_{k-1}+c_{(k-1)(t+1)}x_{t+1}, \notag\\
\rho(x_{k})&=c_{k}+c_{k(t-1)}x_{t-1}, \notag\\
\rho(x_{k+1})&=c_{k+1}+c_{(k+1)t}x_{t} \notag.
\end{align}
Now, if we consider the case $c_{(k-1)(t+1)}=0$, using a similar argument as in the case of $c_{k(t+1)}=0$, we observe that $\rho$ is of the form 
\begin{align}
\rho(x_{k-1})&=c_{k-1}+c_{(k-1)(t-1)}x_{t-1}, \notag\\
\rho(x_{k})&=c_{k}+c_{k(t+1)}x_{t+1}, \notag\\
\rho(x_{k+1})&=c_{k+1}+c_{(k+1)t}x_{t} \notag.
\end{align}
\par \textbf{Case 2.} $c_{k(t-1)}-C_{k+1}c_{(k-1)(t-1)} \neq 0$. Then from equation \eqref{T1} for $i=k+1$, we have $c_{(k+1)(t-1)}=0$. Now, using the argument as in the proof of \textbf{Case 1}, one can observe that $\rho$ can have one of the following two forms:
\begin{align*}
\rho(x_{k-1})&=c_{k-1}+c_{(k-1)(t-1)}x_{t-1},         &  \hspace{0.5cm}&     &   \rho(x_{k-1})&=c_{k-1}+c_{(k-1)(t+1)}x_{t+1},   \\
\rho(x_{k})&=c_{k}+c_{k(t+1)}x_{t+1},     &  \text{or}&   &     \rho(x_{k})&=c_{k}+c_{k(t-1)}x_{t-1},  \\
\rho(x_{k+1})&=c_{k+1}+c_{(k+1)t}x_{t},  &    \hspace{0.5cm}&    & \rho(x_{k+1})&=c_{k+1}+c_{(k+1)t}x_{t}.
\end{align*}
\par
Since $k$ and $t$ are arbitrary, $\rho$ can be written using a permutation of three symbols as;
\begin{align}
&\rho(x_{1})=c_{1}+c_{1\sigma(1)}x_{\sigma(1)}, \notag\\
&\rho(x_{2})=c_{2}+c_{2\sigma(2)}x_{\sigma(2)}, \notag\\
&\rho(x_{3})=c_{3}+c_{3\sigma(3)}x_{\sigma(3)}, \notag
\end{align} 
where $\sigma \in S_{3}$,  and $c_{1\sigma(1)},c_{2\sigma(2)},c_{3\sigma(3)} \in K^{*}$. 

Now, for $ i \in \mathbb{Z}_{3}$, the equation $d\rho(x_{i})=\rho d(x_{i})$ reduces to 
\begin{align}
c_{i\sigma(i)}x_{\sigma(i)}(x_{\sigma(i)-1}-C_{\sigma(i)}x_{\sigma(i)+1})= &(c_{i}+c_{i\sigma(i)}x_{\sigma(i)})\big((c_{i-1}+c_{(i-1)\sigma(i-1)}x_{\sigma(i-1)}) \notag \\&-C_{i}(c_{i+1}+c_{(i+1)\sigma(i+1)}x_{\sigma(i+1)})\big).
\label{T4}    
\end{align}
As $\sigma \in S_{3}$, $x_{\sigma(i)-1} = x_{\sigma(i-1)}$ or $x_{\sigma(i+1)}$. Without loss of generality, we assume that $x_{\sigma(i)-1} = x_{\sigma(i-1)}$. Then $x_{\sigma(i)+1} = x_{\sigma(i+1)}$, and by comparing the coefficient of  $x_{\sigma(i+1)}$, we get $c_{i}=0$ for all $ i \in \mathbb{Z}_{3}$.
Further, equation \eqref{T4} reduces to 
\begin{equation}\label{T4.1} 
 x_{\sigma(i)-1}-C_{\sigma(i)}x_{\sigma(i)+1}=c_{(i-1)\sigma(i-1)}x_{\sigma(i-1)} -C_{i}c_{(i+1)\sigma(i+1)}x_{\sigma(i+1)}. 
\end{equation}
Now, for a given $\sigma \in S_{3}$, the automorphisms that are in the isotropy group of Lotka-Volterra derivation for various values of $C_1, C_2$, and $C_3$ can be computed using equation \eqref{T4.1} and are listed in the table below. 
\end{proof}

\begin{small}
\begin{table}[ht]
\begin{tabular}{ |c|c|c|} 
\hline
  $ \sigma \in S_{3}$  &
$\mathbf{(\rho(x_{1}),\rho(x_{2}),\rho(x_{3}))}$ & Possible values of $C_{1}, C_{2},C_{3}$\\
\hline
  $id$ & $(x_{1},x_{2},x_{3})$ &  $C_{1}, C_{2},C_{3}$ are arbitrary  \\
\hline
  $(12)$ & $(-C_{3}^{-1}x_{2},-C_{3}x_{1},x_{3})$ & $ C_{1}C_{3}=1,C_{2}=-1 $\\
\hline
  $(13)$ & $(-C_{3}^{-1}x_{3},x_{2},-C_{3}x_{1})$ & $C_{2}C_{3}=1, C_{1}=-1 $\\
 \hline
  $(23)$ & $(x_{1},-C_{2}x_{3},-C_{2}^{-1}x_{2})$ &$C_{1}C_{2}=1, C_{3}=-1 $\\
 \hline
   $(123)$ & $(x_{2},x_{3},x_{1})$ & $C_{1}=C_{2}=C_{3} $\\
 \hline
  $(132)$ & $(x_{3},x_{1},x_{2})$ &  $C_{1}=C_{2}=C_{3} $\\
\hline
\end{tabular}
\caption{\textbf{$\rho$ for different values of $\sigma \in S_{3}$}}
    \label{Table3}
\end{table}
\end{small}


\begin{corollary}
Let $d$ be the derivation as defined in Theorem \ref{Main1}. If $C_{1}=C_{2}=C_{3}=-1$, then the isotropy group of $d$ is isomorphic to the dihedral group $D_{6}$ of order $6$.
\end{corollary}
  \begin{proof}
If $C_{1}=C_{2}=C_{3}=-1$, then $1-C_{1}C_{2}C_{3} \neq 0$. Hence, from Table \ref{Table3}, Aut$(K[x_1, x_2, x_3])_d$= $\{(\rho(x_{1}),\rho(x_{2}),\rho(x_{3}))=(x_{1},x_{2},x_{3}),(x_{2},x_{1},x_{3}),(x_{3},x_{2},x_{1}),(x_{1},x_{3},x_{2}),(x_{2},x_{3},x_{1}),(x_{3},x_{1},x_{2})\}$.
\end{proof}

\begin{corollary}
Let $d$ be the derivation as defined in Theorem \ref{Main1}. If $C_{1}C_{2}C_{3}=0$ and $C_{i} \neq 0$ for some $ i \in \mathbb{Z}_{3}$, then the isotropy group of $d$ is trivial.
\end{corollary}
 \begin{proof}
The proof follows from Table \ref{Table3}.
  \end{proof}

\begin{theorem}\label{Main2}
Let $d=x_{1}(x_{3}-C_{1}x_{2})\frac{\partial }{\partial x_{1}}+x_{2}(x_{1}-C_{2}x_{3})\frac{\partial }{\partial x_{2}}+x_{3}(x_{2}-C_{3}x_{1})\frac{\partial }{\partial x_{3}}$, be a derivation of $K[x_{1},x_{2},x_{3}]$. If $1-C_{1}C_{2}C_{3} = 0$, then the isotropy group of $d$ is as listed in Table \ref{Table2}.
\end{theorem}
\begin{proof}
Let $\rho \in  \text{Aut}( K[ x_{1},x_{2},x_{3}])_{d}$. 
Then, for $ i \in \mathbb{Z}_{3}$, we have 
$$\rho(x_{i})=c_{i}+\sum_{j=1}^3 c_{ij}x_{j},$$
where $c_{i},c_{ij} \in K$,  for all $ i, j \in \mathbb{Z}_{3}$.\\
Further, the relation $d\rho\left ( x_{i} \right ) =\rho d\left ( x_{i} \right )$ for $ i  \in \mathbb{Z}_{3}$  gives
\begin{equation}
\begin{aligned}
\sum_{j=1}^3 c_{1j}x_{j}(x_{j-1}-C_{j}x_{j+1})=
 (c_{1}+ \sum_{j=1}^3 c_{1j}x_{j}) 
&\big((c_{3}+\sum_{j=1}^3 c_{3j}x_{j})-C_{1}(c_{2}+\sum_{j=1}^3 c_{2j}x_{j})\big),
\end{aligned}\label{T5}
\end{equation}
\begin{equation} 
\begin{aligned}
\sum_{j=1}^3 c_{2j}x_{j}(x_{j-1}-C_{j}x_{j+1})=(c_{2}+\sum_{j=1}^3 c_{2j}x_{j})
&\big((c_{1}+\sum_{j=1}^3 c_{1j}x_{j})- ~ C_{2}(c_{3}+\sum_{j=1}^3 c_{3j}x_{j})\big), \end{aligned}\label{T6}
\end{equation}
 and
 \begin{equation}
\begin{aligned}
\sum_{j=1}^3 c_{3j}x_{j}(x_{j-1}-C_{j}x_{j+1})=(c_{3}+\sum_{j=1}^3 c_{3j}x_{j}) 
&\big((c_{2}+\sum_{j=1}^3 c_{2j}x_{j})-C_{3}(c_{1}+\sum_{j=1}^3 c_{1j}x_{j})\big). \end{aligned}\label{T7}
\end{equation}

By comparing the coefficient of $x_{1}^{2},~x_{2}^{2}$ and $x_{3}^{2}$ in equations \eqref{T5}-\eqref{T7}, we have
\begin{equation}\label{T8}
\left.\begin{matrix}
 x_{1}^{2} & x_{2}^{2} & x_{3}^{2}\\ 
c_{i1}(c_{(i-1)1}-C_{i}c_{(i+1)1})=0 & c_{i2}(c_{(i-1)2}-C_{i}c_{(i+1)2})=0 &  c_{i3}(c_{(i-1)3}-C_{i}c_{(i+1)3})=0
\end{matrix}\right\}
\end{equation}
where $i \in \mathbb{Z}_{3}$. 

Note that $1-C_{1}C_{2}C_{3} = 0$; this implies $C_{i} \in K^{*}$ for all $i$. Since $\rho$ is an automorphism, 
 $\rho(x_{i-1}-C_{i}x_{i+1}) \notin K$ for all $ i \in \mathbb{Z}_{3}$.
In particular for $i =k$, 
$\rho(x_{k-1}-C_{k}x_{k+1}) 
=(c_{k-1}-C_{k}c_{k+1})+\sum_{j=1}^{3} (c_{(k-1)j}-C_{k}c_{(k+1)j})x_{j} \notin K$. Hence, there exists a $ t \in \mathbb{Z}_{3}$ such that $c_{(k-1)t}-C_{k}c_{(k+1)t} \neq 0$, which implies $c_{kt}=0$,
and by using the coefficient of $x_{t}^{2}$ in equation \eqref{T8} for $i=k+1$, we have $c_{(k-1)t}c_{(k+1)t}=0$. This implies either $c_{(k-1)t}=0$ or $c_{(k+1)t}=0$ but not both (as $\rho$ is an automorphism). 
Without loss of generality,  assume that $c_{(k-1)t}=0$. Then, $c_{(k+1)t} \neq 0$, and
\begin{small}
$$\rho(x_{k}-C_{k+1}x_{k+2})= (c_{k}-C_{k+1}c_{k+2})+(c_{k(t-1)}-C_{k+1}c_{(k+2)(t-1)})x_{t-1}+(c_{k(t+1)}-C_{k+1}c_{(k+2)(t+1)})x_{t+1}$$ 
\end{small}
with, either $c_{k(t-1)}-C_{k+1}c_{(k+2)(t-1)} \neq 0$ or $c_{k(t+1)}-C_{k+1}c_{(k+2)(t+1)} \neq 0$.

\textbf{Case 1.} $c_{k(t+1)}-C_{k+1}c_{(k+2)(t+1)} \neq 0$. Then, from equation \eqref{T8} for $i=k+1$, $c_{(k+1)(t+1)}=0$.
Now, from the coefficient of $x_{t+1}^{2}$ in equation \eqref{T8} for $i=k$, we have $c_{k(t+1)}c_{(k-1)(t+1)}=0$. This implies either $c_{k(t+1)}=0$ or $c_{(k-1)(t+1)}=0$, but not both. \par
\textbf{Subcase 1.} Let $c_{k(t+1)}=0$.  In this case, $\rho$ is given $\rho(x_{k-1})=c_{k-1}+c_{(k-1)(t-1)}x_{t-1}+c_{(k-1)(t+1)}x_{t+1}$, $\rho(x_{k})=c_{k}+c_{k(t-1)}x_{t-1}$, $\rho(x_{k+1})=c_{k+1}+c_{(k+1)(t-1)}x_{t-1}+c_{(k+1)t}x_{t} $. Since $\rho(x_{k}) \notin K$, the constant $c_{k(t-1)} \neq 0$ and therefore from the coefficient of $x_{t-1}^{2}$ in equation \eqref{T8} for $i=k$, we have $c_{k(t-1)}(c_{(k-1)(t-1)}-C_{k}c_{(k+1)(t-1)})=0$. This implies $c_{(k-1)(t-1)}-C_{k}c_{(k+1)(t-1)}=0$. Also
$$\rho(x_{k+1}-C_{k+2}x_{k})= (c_{k+1}-C_{k+2}c_{k})+(c_{(k+1)(t-1)}-C_{k+2}c_{k(t-1)})x_{t-1}-c_{(k+1)t}x_{t}.$$
Now, we have the following two possibilities: 
\par
\textbf{(a)} $c_{(k+1)(t-1)}-C_{k+2}c_{k(t-1)} \neq 0$. In this case, the coefficient of $x_{t-1}^{2}$ in equation \eqref{T8}, for $i=k+2$, gives
$c_{(k+2)(t-1)}=c_{(k-1)(t-1)}=0$. Again, from the coefficient of $x_{t-1}^{2}$ in equation \eqref{T8} for $i=k+1$, we get $c_{k(t-1)}c_{(k+1)(t-1)}=0$. But $c_{k(t-1)}\neq 0$, hence $c_{(k+1)(t-1)}=0$  and  
\begin{align}
\rho(x_{k-1})&=c_{k-1}+c_{(k-1)(t+1)}x_{t+1}, \notag\\
\rho(x_{k})&=c_{k}+c_{k(t-1)}x_{t-1}, \notag\\
\rho(x_{k+1})&=c_{k+1}+c_{(k+1)t}x_{t} \notag.
\end{align}
Now, the relation $d\rho(x_{k-1})=\rho d(x_{k-1})$ reduces to
\begin{align}
  c_{(k-1)(t+1)}x_{t+1}(x_{t}-C_{t+1}x_{t-1})=&(c_{k-1}+c_{(k-1)(t+1)}x_{t+1})\big( (c_{k+1}+c_{(k+1)t}x_{t}) \notag \\ & -C_{k-1}(c_{k}+c_{k(t-1)}x_{t-1}) \big). \notag 
\end{align}
By comparing the coefficient of $x_{t}$, $x_{t+1}$, $x_{t}x_{t+1}$ and $x_{t-1}x_{t+1}$, we have $c_{k-1}=0$, $c_{k+1}=C_{k-1}c_{k}$, $c_{(k+1)t}=1$ and $C_{k-1}c_{k(t-1)}=C_{t+1}$, respectively. Similarly, from $d\rho(x_{k})=\rho d(x_{k})$ and $d\rho(x_{k+1})=\rho d(x_{k+1})$, we have $c_{k}=c_{k+1}=0$, $c_{(k-1)(t+1)}=1$, $c_{k(t-1)}=1$, $C_{k}c_{(k+1)t}=C_{t-1}$ and $C_{k+1}c_{(k-1)(t+1)}=C_{t}$. Thus
$ \forall ~~ k,t \in \mathbb{Z}_{3}$, $C_{k-1}=C_{t+1}$, $C_{k}=C_{t-1}$ and $C_{k+1}=C_{t}$. Hence, $\rho$ is of the form
\begin{align*}
\rho(x_{k-1})&=x_{t+1}, \\
\rho(x_{k})&=x_{t-1}, \\
\rho(x_{k+1})&=x_{t}.
\end{align*}
\par
\textbf{(b)} $c_{(k+1)(t-1)}-C_{k+2}c_{k(t-1)}=0$. We have $c_{(k+1)(t-1)} \neq 0$ as $c_{k(t-1)} \neq 0$. From the coefficient of $x_{t-1}^{2}$ in equation \eqref{T8} for $i=k+1$, we have $c_{(k+1)(t-1)}(c_{k(t-1)}-C_{k+1}c_{(k-1)(t-1)})=0$. This implies $c_{k(t-1)}-C_{k+1}c_{(k-1)(t-1)}=0$ as $c_{(k+1)(t-1)} \neq 0 $. Further, $c_{(k-1)(t-1)} \neq 0$ (as $c_{k(t-1)} \neq 0$) and $\rho$ is of the form 
\begin{align}
\rho(x_{k-1})&=c_{k-1}+c_{(k-1)(t-1)}x_{t-1}+c_{(k-1)(t+1)}x_{t+1}, \notag\\
\rho(x_{k})&=c_{k}+c_{k(t-1)}x_{t-1}, \notag\\
\rho(x_{k+1})&=c_{k+1}+c_{(k+1)(t-1)}x_{t-1}+c_{(k+1)t}x_{t} \notag.
\end{align}

Observe that $c_{(k-1)(t+1)} \neq 0$; otherwise $ \text{Image}(\rho) \subseteq K[x_{t-1}, x_t]$.  

Now, the equation $d\rho(x_{k})=\rho d(x_{k})$ reduces to
\begin{small}
\begin{align}
 c_{k(t-1)}x_{t-1}(x_{t+1}-C_{t-1}x_{t})=(c_{k}+c_{k(t-1)}x_{t-1})\big( (c_{k-1}-C_{k}c_{k+1})+c_{(k-1)(t+1)}x_{t+1}-C_{k}c_{(k+1)t}x_{t} \big).\label{T9}
\end{align}    
\end{small}
By comparing the coefficient of $x_{t+1}$, $x_{t-1}$, $x_{t-1}x_{t+1}$, and $x_{t-1}x_{t}$ in equation \eqref{T9}, we have $c_k=0$,
$c_{k-1}=C_{k}c_{k+1}$, $c_{(k-1)(t+1)}=1$, and $C_{k}c_{(k+1)t}=C_{t-1}$, respectively. \newline
Further, the equation $d\rho(x_{k+1})=\rho d(x_{k+1})$ gives
\begin{align}
c_{(k+1)(t-1)}x_{t-1}(x_{t+1}-C_{t-1}x_{t})+ c_{(k+1)t}x_{t}(x_{t-1}-C_{t}x_{t+1})&=(c_{k+1}+c_{(k+1)(t-1)}x_{t-1}+c_{(k+1)t}x_{t}) \notag \\ &\big(-C_{k+1}c_{k-1}-C_{k+1}x_{t+1} \big),\label{T10}   
\end{align}
and by comparing the coefficient of $x_{t+1}$, $x_{t-1}x_{t}$, $x_{t-1}x_{t+1}$ and $x_{t}x_{t+1}$, we have $c_{k+1}=0$, 
 $c_{(k+1)t}=C_{t-1}c_{(k+1)(t-1)}$, $C_{k+1}=-1$, and $C_{t}=C_{k+1}=-1$, respectively. 
 
Finally,  the relation $d\rho(x_{k-1})=\rho d(x_{k-1})$, along with the previous observations and the condition $C_1C_2C_3=1$, gives $c_{k-1}=c_{k}=c_{k+1}=0$, $c_{(k-1)(t-1)}=c_{(k-1)(t+1)}=c_{(k+1)t}=1$ and $c_{k(t-1)}=c_{(k+1)(t-1)}=-1$. Further, it gives $C_{t-1}=C_{k}=-1$ and $C_{t+1}=C_{k-1}=1$. So, from the relation $C_{t+1}=C_{k-1}=1$, we have $t+1=k-1$. Hence, $t+i=k+i+1$ for all $ i \in \mathbb{Z}_{3}$, 
and $\rho$ is of the form
\begin{align}
\rho(x_{k-1})&=x_{k-1}+x_{k}, \notag\\
\rho(x_{k})&=-x_{k}, \notag\\
\rho(x_{k+1})&=-x_{k}+x_{k+1} \notag.
\end{align}\par
\textbf{Subcase 2.} $c_{(k-1)(t+1)}=0$. Using a similar argument as in \textbf{Subcase 1}, we observe that $\rho$ is of the form  $\rho(x_{k-1})=c_{k-1}+c_{(k-1)(t-1)}x_{t-1}$, $\rho(x_{k})=c_{k}+c_{k(t-1)}x_{t-1}+c_{k(t+1)}x_{t+1}$, $\rho(x_{k+1})=c_{k+1}+c_{(k+1)(t-1)}x_{t-1}+c_{(k+1)t}x_{t} $, where $c_{(k-1)(t-1)},c_{k(t+1)},c_{(k+1)t} \in K^{*}$. Now, we have the following two possibilities: \par
\textbf{(a)} $c_{(k+1)(t-1)}-C_{k+2}c_{k(t-1)} \neq 0$. From the coefficient of $x_{t-1}^{2}$ in equation \eqref{T8}, for $i=k+2$, we have $c_{(k+2)(t-1)}(c_{(k+1)(t-1)} - C_{k+2}c_{k(t-1)})=0$.
Then $c_{(k-1)(t-1)}=c_{(k+2)(t-1)}=0$, which is a contradiction as $\rho(x_{k-1}) \notin K$.\par
\textbf{(b)} $c_{(k+1)(t-1)}-C_{k+2}c_{k(t-1)}=0$. If $c_{k(t-1)}=0$ then $c_{(k+1)(t-1)}=0$. In this case, $C_{k-1}C_{t}=1$, $C_{k}C_{t-1}=1$ and $C_{k+1}C_{t+1}=1$. Hence, $\rho$ is of the form
\begin{align*}
\rho(x_{k-1})=-C_{k+1}^{-1}x_{t-1},~~\rho(x_{k})=-C_{k-1}^{-1}x_{t+1}, ~~ \rho(x_{k+1})&=-C_{k}^{-1}x_{t}.
\end{align*}
If $c_{k(t-1)}\neq 0$, then $c_{(k+1)(t-1)}\neq 0$. In this case,  $t+i=k+i$  $ \forall~~ i \in \mathbb{Z}_{3}$, $C_{k-1}=C_{k}=-1$ and $C_{k+1}=1$. Hence, $\rho$ is of the form
\begin{align*}
\rho(x_{k-1})=x_{k-1},~~\rho(x_{k})=-x_{k-1}+x_{k+1}, ~~ \rho(x_{k+1})=-x_{k-1}+x_{k}.
\end{align*}

\textbf{Case 2.} $c_{k(t-1)}-C_{k+1}c_{(k+2)(t-1)} \neq 0$. Then, from equation \eqref{T8} for $i=k+1$, we have $c_{(k+1)(t-1)}=0$. Using the argument as in the proof of \textbf{Case 1}, one can observe that either $c_{k(t-1)}=0$ or $c_{(k-1)(t-1)}=0$, but not both. \par
\textbf{Subcase 1.} If $c_{k(t-1)}=0$, then $\rho$ can have one of the following two forms:
 \begin{enumerate}
     \item  For $  k,t \in \mathbb{Z}_{3}$, $C_{k-1}C_{t}=1$, $C_{k}C_{t-1}=1$, $C_{k+1}C_{t+1}=1$, and 
\begin{align*}
\rho(x_{k-1})=-C_{k+1}^{-1}x_{t-1},~~ \rho(x_{k})=-C_{k-1}^{-1}x_{t+1},~~ \rho(x_{k+1})=-C_{k}^{-1}x_{t}.
\end{align*}
\item For $k \in \mathbb{Z}_{3}$, $C_{k}=C_{k+1}=-1,~\text{and}~C_{k-1}=1$, and 
\begin{align*}
\rho(x_{k-1})=-x_{k}+x_{k+1},~~\rho(x_{k})=x_{k},~~\rho(x_{k+1})=x_{k-1}+x_{k}.
\end{align*}
\end{enumerate}
\par
\textbf{Subcase 2.} If $c_{(k-1)(t-1)}=0$, then $\rho$ can have one of the following two forms:

\begin{enumerate}
     \item For $ k,t \in \mathbb{Z}_{3}$, $C_{k-1}=C_{t+1}$, $C_{k}=C_{t-1}$ and $C_{k+1}=C_{t}$, and 
\begin{align*}
\rho(x_{k-1})=x_{t+1},~~ \rho(x_{k})=x_{t-1},~~ \rho(x_{k+1})=x_{t}.
\end{align*}
\item For $k \in \mathbb{Z}_{3}$, $C_{k}=C_{k-1}=-1,~\text{and}~C_{k+1}=1$, and 
\begin{align*}
\rho(x_{k-1})=-x_{k-1},~~\rho(x_{k})=-x_{k-1}+x_{k}, ~~ \rho(x_{k+1})=x_{k-1}+x_{k+1}.
\end{align*}
\end{enumerate}
We can compute the isotropy group of $d$ by substituting the values of $k,t \in \mathbb{Z}_{3}$. 
Now, the automorphisms that are in the isotropy group of Lotka-Volterra derivation for various values of $C_1, C_2$, and $C_3$ are shown in the forest below. 

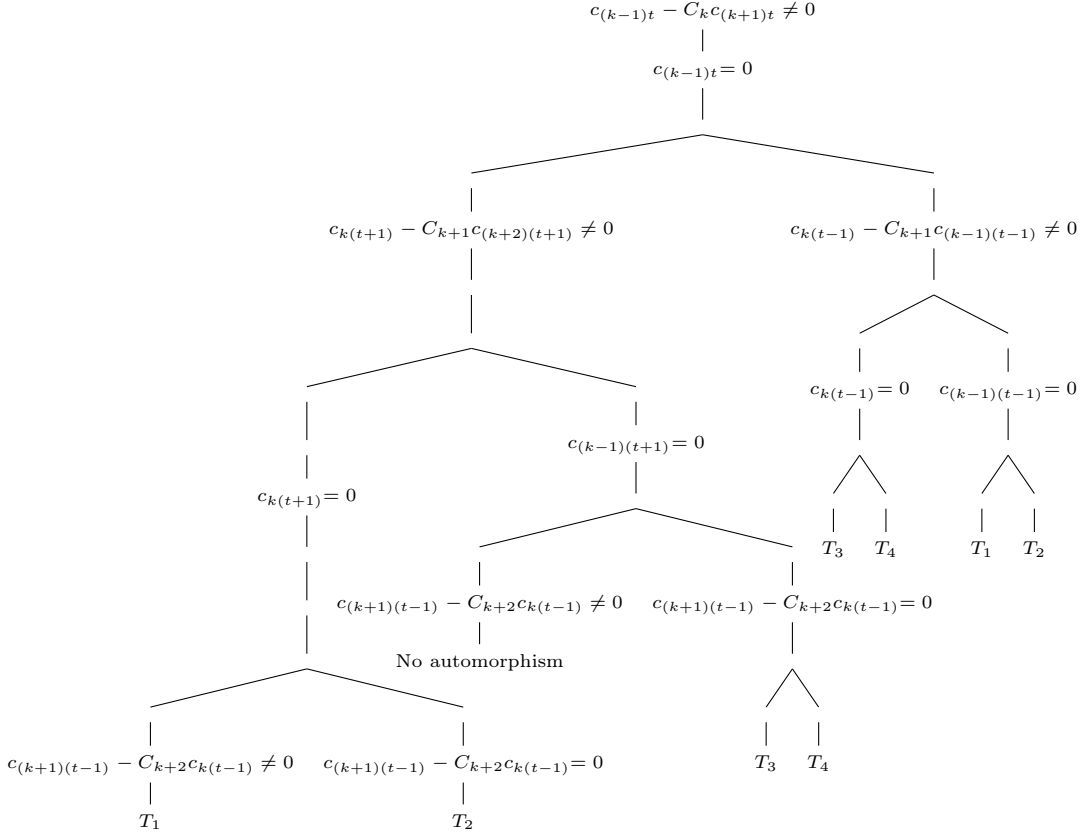
\begin{tiny}
\begin{forest}
[$c_{(k-1)t}-C_{k}c_{(k+1)t} \neq 0$  
   [$c_{(k-1)t} {=0} $ 
      [[[$c_{k(t+1)}-C_{k+1}c_{(k+2)(t+1)} \neq 0$
         [[[[[$c_{k(t+1)}{=0}$ 
            [[[[[$c_{(k+1)(t-1)}-C_{k+2}c_{k(t-1)} \neq 0$[$T_{1}$]]] [[$c_{(k+1)(t-1)}-C_{k+2}c_{k(t-1)} {=0}$[$T_{2}$
            ]]]]]]]]]
[[$c_{(k-1)(t+1)} {=0}$[[[$c_{(k+1)(t-1)}-C_{k+2}c_{k(t-1)} \neq 0$[No automorphism]]] [[$c_{(k+1)(t-1)}-C_{k+2}c_{k(t-1)} {=0}$[[[$T_{3}$]][[$T_{4}$]]
            ]]]
            ]
]
 ]
  ]
   ]
    ]
     ]
    [[$c_{k(t-1)}-C_{k+1}c_{(k-1)(t-1)} \neq 0$
   [[[ $c_{k(t-1)}{=0}$ [[[$T_{3}$]][[$T_{4}$]]]]]  [[$c_{(k-1)(t-1)}{=0}$[[[$T_{1}$]][[$T_{2}$]]]]] 
  ]
   ]
     ]
      ]
      ]
       ]
\end{forest}
\end{tiny}
\vspace{0.2cm}
\begin{center}
The trees $T_{1}$, $T_{2}$, $T_{3}$ and $T_{4}$ are given by    
\end{center}
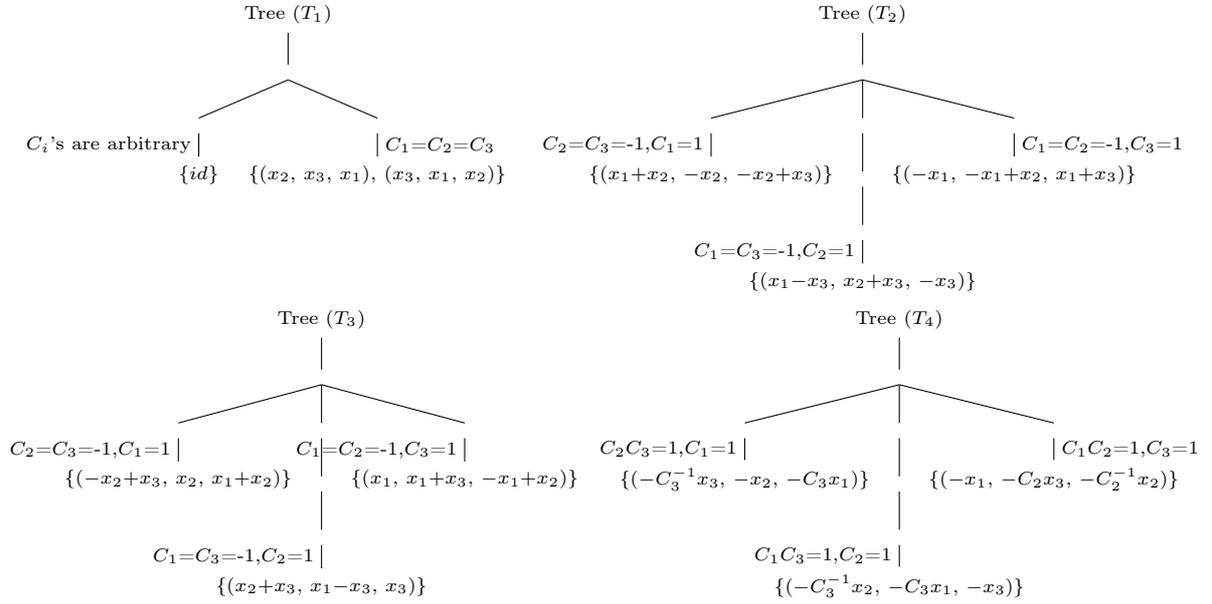
\vspace{0.2cm}
\begin{tiny}
\begin{center}
    \begin{forest}
 [Tree ($T_{1}$)[[[$\{id\}$, edge label={node[midway,left]{$C_{i}$'s are arbitrary}}]]
            [[\{$(x_{2}{,} ~x_{3}{,}~x_{1}) {,}~(x_{3}{,} ~x_{1}{,}~x_{2})$\},edge label={node[midway,right]{$C_{1}${=}$C_{2}${=}$C_{3}$}}]]]]   
\end{forest}
\hspace{0.1cm}
\begin{forest}
 [Tree ($T_{2}$)[[[\{$(x_{1}{+}x_{2}{,} ~{-}x_{2}{,}~{-}x_{2}{+}x_{3})$\},edge label={node[midway, left]{$C_{2}${=}$C_{3}${=-}$1${,}$C_{1}${=}$1$}}]][[[[\{$(x_{1}{-}x_{3}{,} ~x_{2}{+}x_{3}{,}~{-}x_{3})$\},edge label={node[midway, left]{$C_{1}${=}$C_{3}${=-}$1${,}$C_{2}${=}$1$}}]]]][[\{$({-}x_{1}{,} ~{-}x_{1}{+}x_{2}{,}~x_{1}{+}x_{3})$\},edge label={node[midway, right]{$C_{1}${=}$C_{2}${=-}$1${,}$C_{3}${=}$1$}}]]
            ]]   
\end{forest}
\begin{forest}
 [Tree ($T_{3}$)[[[\{$({-}x_{2}{+}x_{3}{,} ~x_{2}{,}~x_{1}{+}x_{2})$\},edge label={node[midway, left]{$C_{2}${=}$C_{3}${=-}$1${,}$C_{1}${=}$1$}}]][[[[\{$(x_{2}{+}x_{3}{,} ~x_{1}{-}x_{3}{,}~x_{3})$\},edge label={node[midway, left]{$C_{1}${=}$C_{3}${=-}$1${,}$C_{2}${=}$1$}}]]]][[\{$(x_{1}{,} ~x_{1}{+}x_{3}{,}~{-}x_{1}{+}x_{2})$\},edge label={node[midway, left]{$C_{1}${=}$C_{2}${=-}$1${,}$C_{3}${=}$1$}}]]
            ]]   
\end{forest}
\begin{forest}
  [Tree ($T_{4}$)[[[\{$({-C_{3}^{-1}}x_{3}{,} ~{-}x_{2}{,}~{-C_{3}}x_{1})$\},edge label={node[midway, left]{$C_{2}C_{3}${=}$1${,}$C_{1}${=}$1$}}]][[[[\{$({-C_{3}^{-1}}x_{2}{,} ~{-C_{3}}x_{1}{,}~{-}x_{3})$\},edge label={node[midway, left]{$C_{1}C_{3}${=}$1${,}$C_{2}${=}$1$}}]]]][[\{$({-}x_{1}{,} ~{-C_{2}}x_{3}{,}~{-C_{2}^{-1}}x_{2})$\},edge label={node[midway, right]{$C_{1}C_{2}${=}$1${,}$C_{3}${=}$1$}}]]]]  
\end{forest}
\end{center}
\end{tiny}
\end{proof} 

\newpage

\begin{corollary}\label{C_{i}=1}
Let $d$ be the derivation as defined in Theorem \ref{Main1}. If $C_{i}=1$ for all $i \in \mathbb{Z}_{3}$, then the isotropy group of $d$ is isomorphic to the dihedral group $D_{6}$ of order $6$.
  \end{corollary}
\begin{proof}
      The proof follows from Table \ref{Table2}.
\end{proof}

\begin{corollary}
Let $d$ be the derivation as defined in Theorem \ref{Main1}. If $C_{1}=C_{2}=C_{3}$ and $C_{1} \neq \pm 1$, then the isotropy group of $d$ is isomorphic to $\mathbb{Z}_{3}$.
  \end{corollary}
  \begin{proof}
      The proof follows from Tables \ref{Table2} and \ref{Table3}.
\end{proof}

\begin{corollary}
Let $d$ be the derivation as defined in Theorem \ref{Main1}. If $C_{k-1}=C_{k}=-1$, and $ C_{k+1}=1$ for some $k \in \mathbb{Z}_{3}$, then the isotropy group of $d$ is isomorphic to $\mathbb{Z}_{2}\times \mathbb{Z}_{2}$.  
\end{corollary}
 \begin{proof}
The proof follows from Table \ref{Table2}.
\end{proof}

\begin{corollary}
Let $d$ be the derivation as defined in Theorem \ref{Main1}. If $C_{k-1}C_{k}=1$, $ C_{k+1}=\pm 1$, and $C_{k} \neq \pm 1$ for some $k \in \mathbb{Z}_{3}$, then the isotropy group of $d$ is isomorphic to $\mathbb{Z}_{2}$.
\end{corollary}
 \begin{proof}
The proof follows from Tables \ref{Table2} and \ref{Table3}.
\end{proof}

\begin{remark}
In Theorems \ref{Main1} and  \ref{Main2}, we have shown that the isotropy group of Lotka-Volterra derivations over 
$K[x_1, x_2, x_3]$ is finite. 
Hence, it is natural to ask whether or not the isotropy group of Lotka-Volterra derivation is finite if there are more than three variables. In the next section, we show that the isotropy group of Lotka-Volterra derivations over 
$K[x_1, x_2, x_3, x_4]$ is not finite in general. 
\end{remark}

\section{\textbf{ Isotropy group of Lotka-Volterra derivation of \texorpdfstring{$K[x_1, x_2, x_3, x_4]$}{} }}

A Lotka-Volterra derivation of $K[x_1, x_2, x_3, x_4]$ is of the form 
\begin{equation}\label{LVD4}
d=x_{1}(x_{4}-C_{1}x_{2})\frac{\partial }{\partial x_{1}}+x_{2}(x_{1}-C_{2}x_{3})\frac{\partial }{\partial x_{2}}+x_{3}(x_{2}-C_{3}x_{4})\frac{\partial }{\partial x_{3}}+x_{4}(x_{3}-C_{4}x_{1})\frac{\partial }{\partial x_{4}}.
\end{equation}
Throughout this section, we study the isotropy group of $d$ defined in equation \eqref{LVD4}. We show that for certain values of $C_{i}$'s, the isotropy group of $d$ is not finite. Later, we also give a complete classification of the isotropy group of $d$ based on the conditions mentioned below.
Consider the sentences:
\begin{enumerate}
    \item[$s_{1}:$] All $C_{i}$'s are $-1$.
    \item[$s_{2}:$] Exactly three $C_{i}$'s are $-1$.
    \item[$s_{3}:$] Exactly two $C_{i}$'s are $-1$.
    \begin{itemize}
        \item[(a)] Consecutive $C_{i}$'s are $-1$.
        \item[(b)] Alternative $C_{i}$'s are $-1$($C_{i}=C_{i+2}=-1$).
        \end{itemize}
    \item[$s_{4}:$] Exactly one $C_{i}$ is $-1$.
    \item[$s_{5}:$] No $C_{i}$ equals $-1$.
 \end{enumerate}
 For $s_1$ and $s_3(b)$ isotropy group of $d$ is not finite, whereas for $s_2, s_3(a), s_4$ and $s_5$ it is finite. 

\begin{proposition}\label{linear4v}
Let $d$ be the derivation as defined in equation \eqref{LVD4} and $\rho \in  \text{Aut}( K[ x_{1},x_{2},x_{3},x_{4}])_{d}$ be such that 
$$\rho(x_{k})=c_{k}+\sum_{j=1}^4 c_{kj}x_{j},\hspace{1cm} \forall ~~ k \in \mathbb{Z}_{4}.$$ 
Then there exist a $t \in \mathbb{Z}_{4}$ such that $c_{kt}=0$ and  $\rho$ is of the form 
\begin{eqnarray}
\rho(x_{k+j})=c_{k+j}+c_{(k+j)(t+j-1)}x_{t+j-1}+c_{(k+j)(t+j+1)}x_{t+j+1} \hspace{1cm} \forall~~ j \in \mathbb{Z}_{4}. \notag
\end{eqnarray}
\end{proposition}
\begin{proof}
We have,  
$$\rho(x_{i})=c_{i}+\sum_{j=1}^4 c_{ij}x_{j},$$
where $c_{i},c_{ij} \in K$,  for all $ i, j \in \mathbb{Z}_{4}$.\\
Then, for $ i \in \mathbb{Z}_{4}$, the relation $d\rho\left ( x_{i} \right ) =\rho d\left ( x_{i} \right )$  gives,
\begin{equation}
\begin{aligned}
\sum_{j=1}^4 c_{ij}x_{j}(x_{j-1}-C_{j}x_{j+1})=
 (c_{i}+ \sum_{j=1}^4 c_{ij}x_{j}) 
&\big((c_{i-1}+\sum_{j=1}^4 c_{(i-1)j}x_{j})-C_{i}(c_{i+1}+\sum_{j=1}^4 c_{(i+1)j}x_{j})\big).
\end{aligned}\label{H1}
\end{equation}

By comparing the coefficient of $x_{j}^{2}$ for all $j \in \mathbb{Z}_{4}$ in equation \eqref{H1}, we have
\begin{equation}\label{H5}
\left.\begin{matrix}
 x_{1}^{2} & x_{2}^{2} \\ 
c_{i1}(c_{(i-1)1}-C_{i}c_{(i+1)1})=0, & c_{i2}(c_{(i-1)2}-C_{i}c_{(i+1)2})=0, \\
& \\
 x_{3}^{2}& x_{4}^{2}\\  c_{i3}(c_{(i-1)3}-C_{i}c_{(i+1)3})=0, & c_{i4}(c_{(i-1)4}-C_{i}c_{(i+1)4})=0.
\end{matrix}\right\}
\end{equation}
\vspace{0.3cm}
Further,  by comparing the coefficient of $x_{j}x_{j+2}$ in equation \eqref{H1}, we have
\begin{eqnarray}
    c_{ij}(c_{(i-1)(j+2)}-C_{i}c_{(i+1)(j+2)})+c_{i(j+2)}(c_{(i-1)j}-C_{i}c_{(i+1)j})=0. \hspace{1cm} \label{H6}
\end{eqnarray}
If $C_{i}=0$ for some $i \in \mathbb{Z}_{4}$, say $i=p$, then from equation \eqref{H6} for $i=p$, we have 
\begin{eqnarray}
    c_{pj}c_{(p-1)(j+2)}+c_{p(j+2)}c_{(p-1)j}=0. \label{2H6} 
\end{eqnarray}
Further, from equation \eqref{H6} for $i=p-1$, we have
\begin{align}
  &c_{(p-1)j}(c_{(p-2)(j+2)}-C_{p-1}c_{p(j+2)})+c_{(p-1)(j+2)}(c_{(p-2)j}-C_{p-1}c_{pj})=0. \notag
 \end{align}
 The above equation can be rewritten as
\begin{align}
  c_{(p-1)j}c_{(p-2)(j+2)}+c_{(p-1)(j+2)}c_{(p-2)j}-C_{p-1}(c_{p(j+2)}c_{(p-1)j}+c_{pj}c_{(p-1)(j+2)})=0, \notag
\end{align}
and by substituting equation \eqref{2H6} in it, we have
\begin{eqnarray}
     c_{(p-1)j}c_{(p-2)(j+2)}+c_{(p-1)(j+2)}c_{(p-2)j}=0.\notag 
\end{eqnarray}
Recursively, we derive 
\begin{eqnarray}
    c_{ij}c_{(i-1)(j+2)}+c_{i(j+2)}c_{(i-1)j}=0 \hspace{1cm} \forall~~ i \in \mathbb{Z}_{4} . \label{H7}
\end{eqnarray}
Also, we know that
$\rho(x_{i-1}-C_{i}x_{i+1}) \notin K$ for all $ i \in \mathbb{Z}_{4}$. 
In particular for $i =k$, 
$\rho(x_{k-1}-C_{k}x_{k+1}) 
=(c_{k-1}-C_{k}c_{k+1})+\sum_{j=1}^{4} (c_{(k-1)j}-C_{k}c_{(k+1)j})x_{j} \notin K$. Hence, there exists a $ t \in \mathbb{Z}_{4}$ such that $c_{(k-1)t}-C_{k}c_{(k+1)t} \neq 0$. Then equation \eqref{H5} for $i=k$, gives $c_{kt}(c_{(k-1)t}-C_{k}c_{(k+1)t} )=0$, which implies, 
$c_{kt}=0$. 
Now, from equation \eqref{H6} for $i=k $ and $j=t$, we have 
\begin{eqnarray}
    c_{k(t+2)}(c_{(k-1)t}-C_{k}c_{(k+1)t})=0. \notag
\end{eqnarray}
This implies $c_{k(t+2)}=0$ and $$\rho(x_{k})=c_{k}+c_{k(t-1)}x_{t-1}+c_{k(t+1)}x_{t+1}.$$
Further, by using the coefficient of $x_{t}^{2}$ in equation \eqref{H5} for $i=k+2$ and $i=k+3(= k-1)$, we get $c_{(k+1)t}c_{(k+2)t}=0$ and $c_{(k+2)t}c_{(k-1)t}=0 $, respectively. 

Since $c_{(k-1)t}-C_{k}c_{(k+1)t} \neq 0$, either $c_{(k-1)t} $ or $c_{(k+1)t} $ is not equal to $0$. 
Without loss of generality, we assume that $c_{(k+1)t} \neq 0$. This implies $c_{(k+2)t}=0$. \newline 
Now, if $C_{k+1} \neq 0$, from equation \eqref{H6} for $i=k+1$ and $j = t $, we have
\begin{eqnarray}
    c_{(k+1)t}(c_{k(t+2)}-C_{k+1}c_{(k+2)(t+2)})+c_{(k+1)(t+2)}(c_{kt}-C_{k+1}c_{(k+2)t})=0, \notag
\end{eqnarray}
which gives $c_{(k+2)(t+2)}=0$  as $c_{kt}=c_{(k+2)t}=c_{k(t+2)}=0$. 

If $C_{k+1} = 0 $, then from equation \eqref{H7}, for $i=k+2$ and $j= t$, we have 
\begin{eqnarray}
    c_{(k+2)t}c_{(k+1)(t+2)}+c_{(k+2)(t+2)}c_{(k+1)t}=0, \notag
\end{eqnarray}
which also shows $c_{(k+2)(t+2)}=0$ and $$\rho(x_{k+2})=c_{k+2}+c_{(k+2)(t-1)}x_{t-1}+c_{(k+2)(t+1)}x_{t+1}.$$
In addition,
\begin{small}
\begin{eqnarray}
\rho(x_{k}-C_{k+1}x_{k+2})= (c_{k}-C_{k+1}c_{k+2})+(c_{k(t-1)}-C_{k+1}c_{(k+2)(t-1)})x_{t-1}+(c_{k(t+1)}-C_{k+1}c_{(k+2)(t+1)})x_{t+1}. \notag
\end{eqnarray}    
\end{small}
As $\rho(x_{k}-C_{k+1}x_{k+2}) \notin K$, either $c_{k(t-1)}-C_{k+1}c_{(k+2)(t-1)} \neq 0$ or $c_{k(t+1)}-C_{k+1}c_{(k+2)(t+1)} \neq 0$.\par
Suppose $c_{k(t+1)}-C_{k+1}c_{(k+2)(t+1)} \neq 0$, then from the coefficient of $x_{t+1}^{2} $ in equation \eqref{H5} for $i=k+1$, we have $c_{(k+1)(t+1)}=0$. Now, from equation \eqref{H6} for $i=k+1$ and $j =t-1$, we have
\begin{eqnarray}
   c_{(k+1)(t-1)}(c_{k(t+1)}-C_{k+1}c_{(k+2)(t+1)})=0. \notag
\end{eqnarray}
The above equation gives $c_{(k+1)(t-1)}=0$ and 
$$\rho(x_{k+1})=c_{k+1}+c_{(k+1)t}x_{t}+c_{(k+2)(t+2)}x_{t+2}.$$
Similarly, we obtain 
$$\rho(x_{k-1})=c_{k-1}+c_{(k-1)t}x_{t}+c_{(k-1)(t+2)}x_{t+2}.$$
When $c_{k(t-1)}-C_{k+1}c_{(k+2)(t-1)} \neq 0$, using a similar argument, one can observe that $\rho(x_{k-1})$ and $\rho(x_{k+1})$ have the same expression as obtained above. \newline
 So, $\rho$ is of the form 
\begin{eqnarray}
\rho(x_{k+j})=c_{k+j}+c_{(k+j)(t+j-1)}x_{t+j-1}+c_{(k+j)(t+j+1)}x_{t+j+1} \hspace{1cm} \forall~~ j \in \mathbb{Z}_{4}. \notag
\end{eqnarray} 
\end{proof}

\begin{theorem}\label{4vT1}
Let $d$ be the derivation as defined in equation \eqref{LVD4}. If any one of the  $s_{3}(a)$, $s_{4}$, or $s_{5}$ holds, then the isotropy group of $d$ is finite.
\end{theorem}
\begin{proof}
Let $\rho \in  \text{Aut}( K[ x_{1},x_{2},x_{3},x_{4}])_{d}$. By Remark \ref{remarkfor4variable} and Lemma \ref{totaldegree}, $\rho(x_{i})$'s are polynomials in $K[x_{1},x_{2},x_{3},x_{4}]$ of total degree $1$. Hence, by Proposition \ref{linear4v}, $\rho$ is of the form
 \begin{eqnarray}
\rho(x_{k+j})=c_{k+j}+c_{(k+j)(t+j-1)}x_{t+j-1}+c_{(k+j)(t+j+1)}x_{t+j+1} \hspace{1cm} \forall~~ j \in \mathbb{Z}_{4}. \notag \label{H8}
\end{eqnarray} 

Now, $d \rho(x_{k+j})= \rho d(x_{k+j})  $ gives 
\begin{eqnarray}\label{H9}
&c_{(k+j)(t+j-1)}x_{t+j-1}(x_{t+j+2}-C_{t+j-1}x_{t+j})+c_{(k+j)(t+j+1)}x_{t+j+1}(x_{t+j}-C_{t+j+1}x_{t+j+2}) \notag \\ & =(c_{k+j}+c_{(k+j)(t+j-1)}x_{t+j-1}+c_{(k+j)(t+j+1)}x_{t+j+1}) \big( (c_{k+j-1} - C_{k+j}c_{k+j+1})+(c_{(k+j-1)(t+j)} \notag \\&- C_{k+j}c_{(k+j+1)(t+j)})x_{t+j}+(c_{(k+j-1)(t+j+2)} - C_{k+j}c_{(k+j+1)(t+j+2)})x_{t+j+2} \big).  
\end{eqnarray}
Since $\rho(x_{k+j-1}-C_{k+j}x_{k+j+1}) \notin K$, this implies either $c_{(k+j-1)(t+j)}- C_{k+j}c_{(k+j+1)(t+j)} \neq 0$ or $c_{(k+j-1)(t+j+2)} - C_{k+j}c_{(k+j+1)(t+j+2)} \neq 0$. By comparing the coefficient of $x_{t+j}$, $x_{t+j+2}$ in equation \eqref{H9}, we have $c_{k+j}(c_{(k+j-1)(t+j)}- C_{k+j}c_{(k+j+1)(t+j)})=0$ and $c_{k+j}(c_{(k+j-1)(t+j+2)} - C_{k+j}c_{(k+j+1)(t+j+2)})=0$. Hence, $c_{k+j}=0$ for all $ j \in \mathbb{Z}_{4}$. 

By comparing the coefficients of  $x_ix_{i+1}$ form  in equation \eqref{H9}, we have
\begin{align}
& x_{t+j}x_{t+j+1}   &&~~ c_{(k+j)(t+j+1)}(c_{(k+j-1)(t+j)}-C_{k+j}c_{(k+j+1)(t+j)})=c_{(k+j)(t+j+1)},  \notag \\ 
& x_{t+j+1}x_{t+j+2} && ~~c_{(k+j)(t+j+1)}(c_{(k+j-1)(t+j+2)}-C_{k+j}c_{(k+j+1)(t+j+2)})=-C_{t+j+1}c_{(k+j)(t+j+1)},\notag\\ 
& x_{t+j+2}x_{t+j-1} && ~~c_{(k+j)(t+j-1)}(c_{(k+j-1)(t+j+2)}-C_{k+j}c_{(k+j+1)(t+j+2)})=c_{(k+j)(t+j-1)},  \notag\\ 
& x_{t+j-1}x_{t+j}   &&~~ c_{(k+j)(t+j-1)}(c_{(k+j-1)(t+j)}-C_{k+j}c_{(k+j+1)(t+j)})=-C_{t+j-1}c_{(k+j)(t+j-1)}.\notag
\end{align}
For $j=0$, $d \rho(x_{k})= \rho d(x_{k})$ gives 
\begin{align}
& x_{t}x_{t+1}   && c_{k(t+1)}(c_{(k-1)t}-C_{k}c_{(k+1)t})=c_{k(t+1)}, \label{H14} \\ 
& x_{t+1}x_{t+2} && c_{k(t+1)}(c_{(k-1)(t+2)}-C_{k}c_{(k+1)(t+2)})=-C_{t+1}c_{k(t+1)},\label{H15}\\ 
& x_{t+2}x_{t-1} && c_{k(t-1)}(c_{(k-1)(t+2)}-C_{k}c_{(k+1)(t+2)})=c_{k(t-1)}, \label{H16}\\ 
& x_{t-1}x_{t}   && c_{k(t-1)}(c_{(k-1)t}-C_{k}c_{(k+1)t})=-C_{t-1}c_{k(t-1)}. \label{H17}    
\end{align}
For $j=1$,  $d \rho(x_{k+1})= \rho d(x_{k+1})  $  gives
\begin{align}
& x_{t+1}x_{t+2}   && c_{(k+1)(t+2)}(c_{k(t+1)}-C_{k+1}c_{(k+2)(t+1)})=c_{(k+1)(t+2)}, \label{H18} \\ 
& x_{t+2}x_{t-1} && c_{(k+1)(t+2)}(c_{k(t-1)}-C_{k+1}c_{(k+2)(t-1)})=-C_{t+2}c_{(k+1)(t+2)},\label{H19}\\ 
& x_{t-1}x_{t} && c_{(k+1)t}(c_{k(t-1)}-C_{k+1}c_{(k+2)(t-1)})=c_{(k+1)t}, \label{H20} \\ 
& x_{t}x_{t+1}   && c_{(k+1)t}(c_{k(t+1)}-C_{k+1}c_{(k+2)(t+1)})=-C_{t}c_{(k+1)t}. \label{H21}    
\end{align}
For $j=2$,  $d \rho(x_{k+2})= \rho d(x_{k+2})  $  gives
\begin{align}
& x_{t+2}x_{t-1}   && c_{(k+2)(t-1)}(c_{(k+1)(t+2)}-C_{k+2}c_{(k-1)(t+2)})=c_{(k+2)(t-1)}, \label{H24} \\ 
& x_{t-1}x_{t} && c_{(k+2)(t-1)}(c_{(k+1)t}-C_{k+2}c_{(k-1)t})=-C_{t-1}c_{(k+2)(t-1)},\label{H25}\\
& x_{t}x_{t+1} && c_{(k+2)(t+1)}(c_{(k+1)t}-C_{k+2}c_{(k-1)t})=c_{(k+2)(t+1)}, \label{H22}\\ 
& x_{t+1}x_{t+2}   && c_{(k+2)(t+1)}(c_{(k+1)(t+2)}-C_{k+2}c_{(k-1)(t+2)})=-C_{t+1}c_{(k+2)(t+1)}.\label{H23}   
\end{align}
For $j=3$,  $d \rho(x_{k-1})= \rho d(x_{k-1})  $  gives
\begin{align}
& x_{t-1}x_{t}   && c_{(k-1)t}(c_{(k+2)(t-1)}-C_{k-1}c_{k(t-1)})=c_{(k-1)t}, \label{H26} \\ 
& x_{t}x_{t+1} && c_{(k-1)t}(c_{(k+2)(t+1)}-C_{k-1}c_{k(t+1)})=-C_{t}c_{(k-1)t},\label{H27}\\ 
& x_{t+1}x_{t+2} && c_{(k-1)(t+2)}(c_{(k+2)(t+1)}-C_{k-1}c_{k(t+1)})=c_{(k-1)(t+2)}, \label{H28}\\ 
& x_{t+2}x_{t-1}   && c_{(k-1)(t+2)}(c_{(k+2)(t-1)}-C_{k-1}c_{k(t-1)})=-C_{t+2}c_{(k-1)(t+2)}.\label{H29}
\end{align}
Recall from Proposition \ref{linear4v}, the integer $t$ is such that $c_{(k-1)t}-C_{k}c_{(k+1)t} \neq 0$.
So, at least one of $c_{(k-1)t}$ or $c_{(k+1)t}$ is not equal to $0$. Without loss of generality, assume $c_{(k+1)t} \neq 0$. 

Then, from equation \eqref{H20}, we have $c_{k(t-1)}-C_{k+1}c_{(k+2)(t-1)}=1$. By substituting this in equation \eqref{H19}, we have $(1+C_{t+2})c_{(k+1)(t+2)}=0$. Also, from equation \eqref{H21}, we have $c_{k(t+1)}-C_{k+1}c_{(k+2)(t+1)}=-C_{t}$. By substituting this in equation \eqref{H18}, we get $(1+C_{t})c_{(k+1)(t+2)}=0$. \par

Further, $$\rho(x_{k}-C_{k+1}x_{k+2}) =(c_{k(t-1)}-C_{k+1}c_{(k+2)(t-1)})x_{t-1}+(c_{k(t+1)}-C_{k+1}c_{(k+2)(t+1)})x_{t+1} \notin K.$$ 
Then either  $c_{k(t-1)}-C_{k+1}c_{(k+2)(t-1)}$ or $c_{k(t+1)}-C_{k+1}c_{(k+2)(t+1)}$ is not equal to $0$. \par
\textbf{Case 1.}  $c_{k(t+1)}-C_{k+1}c_{(k+2)(t+1)} \neq 0$.
Then we have the following two subcases: \par
\textbf{Subcase 1.} $c_{k(t+1)} \neq 0$. In this case, equations \eqref{H14}-\eqref{H29} can be rewritten as, 
\begin{align}
& x_{t}x_{t+1}   && c_{(k-1)t}-C_{k}c_{(k+1)t}=1,\label{H30}  \\ 
&    && c_{k(t+1)}-C_{k+1}c_{(k+2)(t+1)}=-C_{t},\label{H31}  \\
&  && c_{(k+2)(t+1)}(c_{(k+1)t}-C_{k+2}c_{(k-1)t})=c_{(k+2)(t+1)},\label{H32} \\
&  && c_{(k-1)t}(c_{(k+2)(t+1)}-C_{k-1}c_{k(t+1)})=-C_{t}c_{(k-1)t},\label{H33}\\ 
&   && \notag\\
& x_{t+1}x_{t+2} && c_{(k-1)(t+2)}-C_{k}c_{(k+1)(t+2)}=-C_{t+1},\label{H34}\\ 
&   && c_{(k+1)(t+2)}(1+C_{t})=0,\label{H35}\\
&    && c_{(k+2)(t+1)}(c_{(k+1)(t+2)}-C_{k+2}c_{(k-1)(t+2)})=-C_{t+1}c_{(k+2)(t+1)},\label{H36}\\
&  && c_{(k-1)(t+2)}(c_{(k+2)(t+1)}-C_{k-1}c_{k(t+1)})=c_{(k-1)(t+2)},\label{H37} \\ 
& && \notag\\
& x_{t+2}x_{t-1}  && c_{k(t-1)}(1+C_{t+1})=0,\label{H38}\\
&   && c_{(k+1)(t+2)}(1+C_{t+2})=0,\label{H39}\\
&    && c_{(k+2)(t-1)}(c_{(k+1)(t+2)}-C_{k+2}c_{(k-1)(t+2)})=c_{(k+2)(t-1)}, \label{H40} \\
&   && c_{(k-1)(t+2)}(c_{(k+2)(t-1)}-C_{k-1}c_{k(t-1)})=-C_{t+2}c_{(k-1)(t+2)},\label{H41}\\
& && \notag\\
& x_{t-1}x_{t}   && c_{k(t-1)}(1+C_{t-1})=0,\label{H42}\\
&  && c_{k(t-1)}-C_{k+1}c_{(k+2)(t-1)}=1,\label{H43}  \\ 
&  && c_{(k+2)(t-1)}(c_{(k+1)t}-C_{k+2}c_{(k-1)t})=-C_{t-1}c_{(k+2)(t-1)},\label{H44}\\
&    && c_{(k-1)t}(c_{(k+2)(t-1)}-C_{k-1}c_{k(t-1)})=c_{(k-1)t}.\label{H45}  
\end{align}
If any of the $s_{3}(a)$ or $s_{4}$ or $s_{5}$ holds, then in each case, from equations \eqref{H35}, \eqref{H38}, \eqref{H39}, and \eqref{H42}, we have $c_{(k+1)(t+2)}=c_{k(t-1)}=0$. So, $\rho(x_{k})=c_{k(t+1)}x_{t+1}$, $\rho(x_{k+1})=c_{(k+1)t}x_{t}$, $\rho(x_{k-1})=c_{(k-1)t}x_{t}+c_{(k-1)(t+2)}x_{t+2}$ and $\rho(x_{k+2})=c_{(k+2)(t-1)}x_{t-1}+c_{(k+2)(t+1)}x_{t+1}$. \newline
As $\rho$ is an automorphism, we have $c_{(k-1)(t+2)} \neq 0$ and $c_{(k+2)(t-1)} \neq 0$. Also, $C_{i} \neq -1$ for at least one $i$.
Without loss of generality, we assume $C_{t} \neq -1$.
Now, equations \eqref{H33} and \eqref{H37} together with $C_{t} \neq -1$ give $c_{(k-1)t}=0$.
So, $\rho(x_{k-1})=c_{(k-1)(t+2)}x_{t+2}$. \par
Now, from equation \eqref{H44}, we have $c_{(k+1)t}=-C_{t-1}$. Further, equation $\eqref{H32}$  $c_{(k+2)(t+1)}(1+C_{t-1})=0$. 
Similarly, from equation \eqref{H40} and \eqref{H36}, $c_{(k+2)(t+1)}(1+C_{t+1})=0$. These two together give  $c_{(k+2)(t+1)}=0$ as one of $1+C_{t+1}$ or $1+C_{t-1}$ is non-zero. So, $\rho(x_{k+2})=c_{(k+2)(t-1)}x_{t-1}$.

Therefore, $\rho$ is of the form 
$$\rho(x_{k+j})=c_{(k+j)(t-j+1)}x_{t-j+1} \hspace{1cm} \forall~~ j \in \mathbb{Z}_{4}.$$ 

\textbf{Subcase 2.} $c_{(k+2)(t+1)} \neq 0$. With the similar argument as in \textbf{Subcase 1}, we can show that $$\rho(x_{k+j})=c_{(k+j)(t+j-1)}x_{t+j-1} \hspace{1cm} \forall ~~j \in \mathbb{Z}_{4}.$$
\par
\textbf{Case 2.} $c_{k(t-1)}-C_{k+1}c_{(k+2)(t-1)} \neq 0$. Using a similar argument as in \textbf{Case 1},  if $c_{k(t-1)} \neq 0$, then 
$\rho$ is of the form $$\rho(x_{k+j})=c_{(k+j)(t+j-1)}x_{t+j-1} \hspace{1cm} \forall~~ j \in \mathbb{Z}_{4},$$
and if $c_{(k+2)(t-1)} \neq 0$, then $\rho$ is of the form $$\rho(x_{k+j})=c_{(k+j)(t-j+1)}x_{t-j+1} \hspace{1cm} \forall~~ j \in \mathbb{Z}_{4}.$$
\par

\vspace{.2cm}
\noindent Now, we want to compute the relations among $C_{i}$'s. \\When $\rho$ is of the form 
$$\rho(x_{k+j})=c_{(k+j)(t-j+1)}x_{t-j+1} \hspace{1cm} \forall~~ j \in \mathbb{Z}_{4}.$$ 
The relation  $d\rho(x_{k+j})=\rho d (x_{k+j})$ gives
\begin{align}
c_{(k+j)(t-j+1)}x_{t-j+1}(x_{t-j} - C_{t-j+1}x_{t-j+2})= c_{(k+j)(t-j+1)} & x_{t-j+1} (c_{(k+j-1)(t-j+2)}x_{t-j+2}\notag\\& -C_{k+j}c_{(k+j+1)(t-j)}x_{t-j}) \notag.
\end{align}
By comparing the coefficients $x_{t-j}x_{t-j+1}$ and $x_{t-j+1}x_{t-j+2}$ in the above equation, we have $-C_{k+j}c_{(k+j+1)(t-j)}=1$ and $c_{(k+j-1)(t-j+2)}=- C_{t-j+1}$, respectively. Thus, for $ k,t \in \mathbb{Z}_{4}$, $C_{k+j-1}C_{t-j}=1$, and $\rho(x_{k+j})=-C_{t-j}x_{t-j+1}$ for all $j \in \mathbb{Z}_{4}$.\newline \par
When $\rho$ is of the form 
$$\rho(x_{k+j})=c_{(k+j)(t+j-1)}x_{t+j-1} \hspace{1cm} \forall ~~j \in \mathbb{Z}_{4}.$$ The relation $d\rho(x_{k+j})=\rho d (x_{k+j})$ gives
\begin{align}
c_{(k+j)(t+j-1)}x_{t+j-1}(x_{t+j-2} - C_{t+j-1}x_{t+j})= c_{(k+j)(t+j-1)}&x_{t+j-1} (c_{(k+j-1)(t+j-2)}x_{t+j-2}\notag\\& - C_{k+j}c_{(k+j+1)(t+j)}x_{t+j}) \notag.
\end{align}
By comparing the coefficients $x_{t+j-1}x_{t+j-2}$ and $x_{t+j-1}x_{t+j}$ in the above equation, we have $c_{(k+j-1)(t+j-2)}=1$ and $C_{k+j}c_{(k+j+1)(t+j)}= C_{t+j-1}$, respectively. Further, $c_{(k+j+1)(t+j)}=1$ and $C_{k+j}= C_{t+j-1}$. Thus, for $  k,t \in \mathbb{Z}_{4}$, $C_{k+j}= C_{t+j-1}$, and $\rho$ is given by
$\rho(x_{k+j})=x_{t+j-1}$ for all $j \in \mathbb{Z}_{4}$.\newline
Hence, in each case, the isotropy group of $d$ is finite.
\end{proof}

\begin{corollary}\label{4C_{i}=1}
 Let $d$ be the derivation as defined in Theorem \ref{4vT1}. If $C_{i}=1$ for all $i\in \mathbb{Z}_{4}$, then the isotropy group of $d$ is isomorphic to the dihedral group $D_{8}$ of order $8$.  
\end{corollary}
\begin{proof}
 Let $\rho \in Aut( K[x_{1}, x_{2}, x_{3}, x_{4}]  )_{d}$. Since $C_{i}=1$ for all $i$, $s_{5}$ holds. From Theorem \ref{4vT1}, we have $\rho(x_{k+j})=-x_{t-j+1}$ and $\rho(x_{k+j})=x_{t+j-1}$ for all $j,k,t \in \mathbb{Z}_{4}$. Then, the isotropy group of $d$ is 
\begin{align}
\text{Aut}(K[x_{1}, x_{2}, x_{3}, x_{4}])_d = \{&(\rho(x_{1}),\rho(x_{2}),\rho(x_{3}),\rho(x_{4}))= (x_{1},x_{2},x_{3},x_{4}), (x_{2},x_{3},x_{4},x_{1}),\notag\\
&  ~~~~~~~~ (x_{3},x_{4},x_{1},x_{2}),(x_{4},x_{1},x_{2},x_{3}),
  (-x_{3},-x_{2},-x_{1},-x_{4}),\notag\\
&(-x_{4},-x_{3},-x_{2},-x_{1}), (-x_{1},-x_{4},-x_{3},-x_{2}),(-x_{2},-x_{1},-x_{4},-x_{3}) \}. \notag  
\end{align}
\end{proof}

\begin{theorem}
Let $d$ be the derivation as defined in equation \eqref{LVD4}. If $s_{2}$ holds, then the isotropy group of $d$ is isomorphic to $\mathbb{Z}_{2}$. \end{theorem}
\begin{proof}
As $s_{2}$ holds, exactly three $C_{j}$'s are $-1$. Let  $C_{i} \neq -1$ and $C_{i+1}=C_{i+2}=C_{i+3}=-1 $ for some $i \in \mathbb{Z}_4$. Let $\rho \in  \text{Aut}( K[ x_{1},x_{2},x_{3},x_{4}])_{d}$.
Then, by Remark \ref{remarkfor4variable}, we know that $\rho(x_{i})$ is a polynomial in $K[x_{1},x_{2},x_{3},x_{4}]$ of total degree $1$. Also, by Lemma \ref{totaldegree} $\rho(x_{i}+x_{i+2})$ is a polynomial in $K[x_{1},x_{2},x_{3},x_{4}]$ of total degree $1$. This implies $\rho(x_{i+2})$ is a polynomial in $K[x_{1},x_{2},x_{3},x_{4}]$ of total degree $1$.

Again, from Lemma \ref{totaldegree} $\rho(x_{i-1}-C_{i}x_{i+1})$, and $\rho(x_{i+1}+x_{i-1})$ are polynomials in $K[x_{1},x_{2},x_{3},x_{4}]$ of total degree $1$. Let $\rho(x_{i-1})-C_{i}\rho(x_{i+1})={p_{0}}+{p_{1}}$ and $\rho(x_{i+1})+\rho(x_{i-1})=h_{0}+h_{1}$, where $h_{0},p_{0} \in K$ and $h_{1}, p_{1}$  are homogeneous polynomials of degree $1$. This gives $(1+C_{i})\rho(x_{i+1})=h_{0}+h_{1}-(p_{0}+p_{1})$. As $C_{i} \neq -1$, $\rho(x_{i+1})$ is also a polynomial of total degree 1, which further implies $\rho(x_{i-1})$ is a polynomial of total degree 1. Hence, for all $i$, $\rho(x_{i})$ is a polynomial of total degree $1$, and from Proposition \ref{linear4v}, $\rho$ is of the form
 \begin{eqnarray}
\rho(x_{k+j})=c_{k+j}+c_{(k+j)(t+j-1)}x_{t+j-1}+c_{(k+j)(t+j+1)}x_{t+j+1} \hspace{1cm} \forall~~ j \in \mathbb{Z}_{4}.\notag
\end{eqnarray}  
Now, $d \rho(x_{k+j})= \rho d(x_{k+j})$ gives equation \eqref{H9}. We proceed with a similar argument as in Theorem \ref{4vT1}. Let $c_{(k+1)t} \neq 0$,
and $$\rho(x_{k}-C_{k+1}x_{k+2}) =(c_{k(t-1)}-C_{k+1}c_{(k+2)(t-1)})x_{t-1}+(c_{k(t+1)}-C_{k+1}c_{(k+2)(t+1)})x_{t+1} \notin K.$$ 
Then either  $c_{k(t-1)}-C_{k+1}c_{(k+2)(t-1)}$ or $c_{k(t+1)}-C_{k+1}c_{(k+2)(t+1)}$ is not equal to $0$. \par
\textbf{Case 1.}  $c_{k(t+1)}-C_{k+1}c_{(k+2)(t+1)} \neq 0$.
Then we have the following two subcases:
\par
\textbf{Subcase 1.} $c_{k(t+1)} \neq 0$. In this case, we have equations \eqref{H30}-\eqref{H45}. \newline 
Without loss of generality, we assume $C_{t} \neq -1$. From equation \eqref{H35},  $c_{(k+1)(t+2)}=0$ and $\rho(x_{k+1})=c_{(k+1)t}x_{t}$. Further, $\rho(x_{k-1})=c_{(k-1)t}x_{t}+c_{(k-1)(t+2)}x_{t+2}$ gives  $c_{(k-1)(t+2)} \neq 0$ (as $\rho$ is an automorphism). By equation \eqref{H37}, $c_{(k+2)(t+1)}-C_{k-1}c_{k(t+1)}=1$, which further from equation \eqref{H33} gives $c_{(k-1)t}=0$. \par
Since $C_{t} \neq -1$, by $s_2$, $C_{t-1}=C_{t+1}=C_{t+2}=-1$. Hence, from equation \eqref{H34},  $c_{(k-1)(t+2)}=1$ and $\rho(x_{k-1})=x_{t+2}$. Also, $\rho(x_{k+2})=c_{(k+2)(t-1)}x_{t-1}+c_{(k+2)(t+1)}x_{t+1}$. Then either $c_{(k+2)(t-1)}$ or $c_{(k+2)(t+1)}$ is non equal to zero. So, from equations \eqref{H32} and \eqref{H44}, we have $c_{(k+1)t}=1$ as $c_{(k-1)t}=0$ and $\rho(x_{k+1})=x_{t}$. Similarly, from equations \eqref{H36} and \eqref{H40}, we have $C_{k+2}=-1$ as $c_{(k-1)(t+2)}=1$. Thus, from equations \eqref{H41} and \eqref{H43}, we have $c_{(k+2)(t-1)}-C_{k-1}c_{k(t-1)}=1$ and $c_{k(t-1)}-C_{k+1}c_{(k+2)(t-1)}=1$, respectively. \par
As $c_{(k+1)t}=1$, from equation \eqref{H30}, we have $C_{k}=-1$. Further, either $C_{k-1} =-1$ or $C_{k+1}=-1 $ but not both. \newline 
If $C_{k-1} =-1$ and $C_{k+1} \neq -1$, $k+1=t$, and the equations $c_{(k+2)(t-1)}-C_{k-1}c_{k(t-1)}=1$ and $c_{k(t-1)}-C_{k+1}c_{(k+2)(t-1)}=1$ together gives $c_{(k+2)(t-1)}=0$ and $c_{k(t-1)}=1$. So, $\rho(x_{k})=x_{t-1}+c_{k(t+1)}x_{t+1}$ and $\rho(x_{k+2})=c_{(k+2)(t+1)}x_{t+1}$. Also, $c_{(k-1)(t+2)}=1$, so, from equations \eqref{H31} and \eqref{H37}, we have $c_{k(t+1)}-C_{k+1}c_{(k+2)(t+1)}=-C_{t}$ and $c_{(k+2)(t+1)}-C_{k-1}c_{k(t+1)}=1$. This together gives $c_{k(t+1)}=0$, which is not possible. 
Hence, $C_{k+1} = -1$ and $C_{k-1} \neq -1 $ and $t=k-1$. Further, equations $c_{(k+2)(t-1)}-C_{k-1}c_{k(t-1)}=1$ and $c_{k(t-1)}-C_{k+1}c_{(k+2)(t-1)}=1$ together gives $c_{k(t-1)}=0$ and $c_{(k+2)(t-1)}=1$. So, $\rho(x_{k})=c_{k(t+1)}x_{t+1}$ and $\rho(x_{k+2})=x_{t-1}+c_{(k+2)(t+1)}x_{t+1}$. 
\par 
Further, $c_{k(t+1)}=-1$ and $c_{(k+2)(t+1)}=1-C_{t}$, and $\rho$ is of the form
\begin{equation}\label{4T3}
\rho(x_{k-1})=x_{k+1}, ~~\rho(x_{k})=-x_{k},~~ \rho(x_{k+1})=x_{k-1},~~\rho(x_{k+2})=x_{k+2}+(1-C_{k-1})x_{k}.    
\end{equation}

\textbf{Subcase 2.} If $ c_{(k+2)(t+1)} \neq 0$. Using a similar argument, we can show that if $C_{t} \neq -1$, we have $C_{k+2}=C_{k}=-1$, and  there are two possible combinations: 
\begin{enumerate}
    \item $C_{k-1}=-1$ and $C_{k+1} \neq -1$, 
    \item $C_{k+1}=-1$ and $C_{k-1}  \neq -1 $.
\end{enumerate}
In the first choice, $t=k+1$, and $\rho$ is an identity automorphism.   \newline
In the second choice, $t=k-1$, and $\rho$ is of the form 
$$\rho(x_{k-1})=x_{k+1}, ~~\rho(x_{k})=-x_{k},~~ \rho(x_{k+1})=x_{k-1}, ~~\rho(x_{k+2})=x_{k+2}+(1-C_{k-1})x_{k}.$$ \par
Further, if $C_{t+1}$ or $C_{t+2}$ is not equal to $-1$, then $\rho$ is an identity automorphism. \par
If $C_{t-1} \neq -1$, then  $\rho$ is either an identity automorphism or  given by equation \eqref{4T3}. \par
\textbf{Case 2.} $c_{k(t-1)}-C_{k+1}c_{(k+2)(t-1)} \neq 0$. Using a similar argument as in \textbf{Case 1}, we observe that $\rho$ is either an identity automorphism or given by equation \eqref{4T3}.\newline 
Hence, we conclude that the isotropy group of $d$ is isomorphic to $\mathbb{Z}_{2}$.
\end{proof}

Now, we show that in the case of $s_{3}(b)$, the isotropy group of $d$ is infinite. 

\begin{theorem}\label{T4.2}
Let $d$ be the derivation as defined in equation \eqref{LVD4}. If $s_{3}(b)$ holds, then the isotropy group of $d$ is not finite.
\end{theorem}
\begin{proof}
For $C_{1}=C_{3}=1$ and $C_{2}=C_{4}=-1$, consider an automorphism $\rho$ given by:
$$\rho(x_{1})=x_{3},~~ \rho(x_{2})=\alpha x_{2}+\beta x_{4},~~
 \rho(x_{3})=x_{1},~~\rho(x_{4})=(\alpha+1) x_{2}+(\beta -1) x_{4},$$ 
where $\alpha,\beta \in K$. It is easy to observe that $\rho \in \text{Aut}(K[X])_{d}$. Since $K$ is infinite, the isotropy group of $d$ contains infinitely many elements. 
\end{proof}

In Theorem \ref{T4.2}, we observed that the isotropy group of Lotka-Volterra derivation $d$, defined in equation \eqref{LVD4}, need not be finite.  In general, the isotropy group of $d$ may contain an automorphism $\rho$ such that the $\rho(x_{i})$'s are not linear polynomials in $K[x_1, x_2, x_3, x_4]$. When $s_{1}$ holds, we show that the isotropy group of $d$ is infinite and it contains non-linear elements. Next two propositions help us to compute non-linear automorphisms in the isotropy group of $d$.

\begin{proposition}\label{4variable}
Let $d$ be the derivation as defined in equation \eqref{LVD4} and let $\rho \in \text{Aut}(K[X])_{d}$.  If $s_{1}$ holds, then $\rho(x_{1}+x_{3})=x_{1}+x_{3}$ and $\rho(x_{2}+x_{4})=x_{2}+x_{4}$ or $\rho(x_{1}+x_{3})= x_{2}+x_{4}$ and $\rho(x_{2}+x_{4})=x_{1}+x_{3}$. 
\end{proposition}
\begin{proof}
From Lemma \ref{totaldegree}, $\rho(x_{1}+x_{3})$ and $\rho(x_{2}+x_{4})$ are polynomials of total degree $1$. Let $\rho(x_{1}+x_{3})= a_{0}+\sum_{j=1}^4 a_{j}x_{j}$, and $\rho(x_{2}+x_{4})=b_{0}+\sum_{j=1}^4 b_{j}x_{j}$ where $a_{i},b_{i} \in K$ for $0 \leq i \leq 4$. 
Then, $d(\rho(x_{1}+x_{3}))=\rho(d(x_{1}+x_{3}))$ gives 
\begin{align*}
 d( a_{0}+\sum_{j=1}^4 a_{j}x_{j})= \rho((x_{1}+x_{3})(x_{2}+x_{4})).
 \end{align*}
This implies 
 \begin{align}\label{4variable1}
(a_{1}x_{1}+a_{3}x_{3})(x_{2}+x_{4})+(a_{2}x_{2}+a_{4}x_{4})(x_{1}+x_{3})=( a_{0}+\sum_{j=1}^4 a_{j}x_{j})( b_{0}+\sum_{j=1}^4 b_{j}x_{j}).
        \end{align}
By comparing the coefficient of $x_{i}^{2}$ in the above equation, we have $a_{i}b_{i}=0$ for $0 \leq i \leq 4$. Since $\rho$ is an automorphism, $a_{i} \neq 0$ for some $i \geq 1$. Let $a_{1} \neq 0$. Then $b_{1}=0$, and by comparing the coefficients of $x_{1}$ and $x_{1}x_{3}$, we have $b_{0}=0$ and $b_{3}=0$, respectively. So, $\rho(x_{2}+x_{4})=b_{2}x_{2}+b_{4}x_{4}$. \\
Similarly, by comparing the coefficients of $x_{2}$ and $x_{4}$, we have $a_{0}b_{2}=0$ and $a_{0}b_{4}=0$, respectively. As $b_{2}$ and $b_{4}$ cannot be zero together, $a_{0}=0$. By comparing the coefficients of $x_{2}^{2}$, $x_{4}^{2}$, and $x_{2}x_{4}$, we get $a_{2}=a_{4}=0$, and equation \eqref{4variable1} reduces to
\begin{center}
    $(a_{1}x_{1}+a_{3}x_{3})(x_{2}+x_{4})=(a_{1}x_{1}+a_{3}x_{3})(b_{2}x_{2}+b_{4}x_{4})$.
\end{center}
The above equation gives $b_{2}=b_{4}=1$. Now, the relation 
$d(\rho(x_{2}+x_{4}))=\rho(d(x_{2}+x_{4}))$ gives 
\begin{center}
    $(x_{1}+x_{3})(x_{2}+x_{4})=(a_{1}x_{1}+a_{3}x_{3})(x_{2}+x_{4})$.
\end{center}
Thus, we have $a_{1}=a_{3}=1$ and  $\rho(x_{1}+x_{3})=x_{1}+x_{3}$ and $\rho(x_{2}+x_{4})=x_{2}+x_{4}$. 
If $a_{3} \neq 0$, we obtain the same $\rho$. \newline 
Similarly, one can show that if either $a_{2}$ or $a_{4}$ is non-zero, then $\rho(x_{1}+x_{3})= x_{2}+x_{4}$ and $\rho(x_{2}+x_{4})=x_{1}+x_{3}$.
    \end{proof}

\begin{proposition}\label{propinf}
Let $d$ be the derivation as defined in equation \eqref{LVD4} and let $\rho \in \text{Aut}(K[X])_{d}$. Let $s_1$ hold and $\rho(x_{1}+x_{3})=x_{1}+x_{3}$, $\rho(x_{2}+x_{4})=x_{2}+x_{4}$.
 Then $\rho(x_{2}) \in K[x_1+x_3, x_2, x_4]$.\newline 
 Furthermore, 
$\rho(x_{2})=x_{2} \overline{c_{20}}+\sum_{j=1}^{n_{2}} c_{2j}x_{4}^{j}$, where $\overline{c_{20}},c_{2j} \in K[x_{2},x_{1}+x_{3}] $ for all $j\geq 1$  and $x_{2} \nmid \overline{c_{20}}$.
\end{proposition}
\begin{proof}
Let $\rho(x_{2})=\sum_{j=0}^{n_{2}} c_{2j}x_{4}^{j}$, where $c_{2j} \in K[x_{1},x_{2},x_{3}] $ for all $j\geq 0$. 
The relation $d\rho(x_{2})= \rho d(x_{2})$ gives 
\begin{eqnarray}
    &x_{1}(x_{2}+x_{4})(\frac{\partial c_{20}}{\partial x_{1}}+\cdots+\frac{\partial c_{2n_{2}}}{\partial x_{1}}x_{4}^{n_{2}})+x_{2}(x_{1}+x_{3})(\frac{\partial c_{20}}{\partial x_{2}}+\cdots+\frac{\partial c_{2n_{2}}}{\partial x_{2}}x_{4}^{n_{2}})\notag \\ \notag& +x_{3}(x_{2}+x_{4})(\frac{\partial c_{20}}{\partial x_{3}}+\cdots+\frac{\partial c_{2n_{2}}}{\partial x_{3}}x_{4}^{n_{2}})+x_{4}(x_{1}+x_{3})(c_{21}+\cdots+n_{2}c_{2n_{2}}x_{4}^{n_{2}-1}) 
\notag \\&=(c_{20}+\cdots+ c_{2n_{2}}x_{4}^{n_{2}})(x_{1}+x_{3}). \label{4LV}
\end{eqnarray}
By comparing the constant term in equation \eqref{4LV}, we have 
\begin{eqnarray}
    x_{1}x_{2}\frac{\partial c_{20}}{\partial x_{1}}+x_{2}(x_{1}+x_{3})\frac{\partial c_{20}}{\partial x_{2}}+x_{2}x_{3}\frac{\partial c_{20}}{\partial x_{3}}=c_{20}(x_{1}+x_{3}). \label{4LV1}
\end{eqnarray}
Observe that $x_{2} \mid L.H.S.$, this implies $x_{2} \mid c_{20}$. Let $c_{20}=x_{2}^{r}\overline{c_{20}}$, where $\overline{c_{20}} \in K[x_{1},x_{2},x_{3}]$ and $x_{2} \nmid \overline{c_{20}}$. By substituting this $c_{20}$ in equation \eqref{4LV1}, we have 
\begin{eqnarray}
    x_{1}x_{2}\frac{\partial (x_{2}^{r}\overline{c_{20}})}{\partial x_{1}}+x_{2}(x_{1}+x_{3})\frac{\partial (x_{2}^{r}\overline{c_{20}})}{\partial x_{2}}+x_{2}x_{3}\frac{\partial (x_{2}^{r}\overline{c_{20}})}{\partial x_{3}}=(x_{2}^{r}\overline{c_{20}})(x_{1}+x_{3}). \notag
\end{eqnarray}
This can be further reduced to
\begin{eqnarray}
 x_{2}(x_{1}\frac{\partial \overline{c_{20}}}{\partial x_{1}}+x_{3}\frac{\partial \overline{c_{20}}}{\partial x_{3}})+(x_{1}+x_{3})(x_{2}\frac{\partial \overline{c_{20}}}{\partial x_{2}}+r\overline{c_{20}})=(x_{1}+x_{3})\overline{c_{20}}. \label{4LV2}   
\end{eqnarray}
For $x_{2}=0$, the above equation gives $r\overline{c_{20}}(x_{1},0,x_{3})=\overline{c_{20}}(x_{1},0,x_{3})$. Hence, $r=1$ as $x_{2} \nmid \overline{c_{20}}$. This implies $c_{20}=x_{2}\overline{c_{20}}$, and from equation \eqref{4LV2}, we have 
\begin{eqnarray}
 x_{1}\frac{\partial \overline{c_{20}}}{\partial x_{1}}+x_{3}\frac{\partial \overline{c_{20}}}{\partial x_{3}}+(x_{1}+x_{3})\frac{\partial \overline{c_{20}}}{\partial x_{2}}=0. \label{4LV3}  
\end{eqnarray}
\textbf{Claim 1:}   $\overline{c_{20}} \in K[x_{2},x_{1}+x_{3}]$. \\
Let $\overline{c_{20}}=\sum_{j=0}^{s} f_{j}x_{2}^{j}$, where $s \geq 0$,  $f_{j} \in K[x_{1},x_{3}]$ and $f_{s} \neq 0$. By substituting this $\overline{c_{20}}$ in equation \eqref{4LV3}, we have
\begin{eqnarray}
x_{1}(\frac{\partial f_{0}}{\partial x_{1}}+\cdots+\frac{\partial f_{s}}{\partial x_{1}}x_{2}^{s}) +x_{3}(\frac{\partial f_{0}}{\partial x_{3}}+\cdots+\frac{\partial f_{s}}{\partial x_{3}}x_{2}^{s})+(x_{1}+x_{3})(f_{1}+\cdots+sf_{s}x_{2}^{s-1})=0. \label{4LV4} 
\end{eqnarray}
By comparing the coefficient of $x_{2}^{i}$ in equation \eqref{4LV4}, we have 
\begin{eqnarray}
& x_{1}\frac{\partial f_{s}}{\partial x_{1}}+x_{3}\frac{\partial f_{s}}{\partial x_{3}}=0 \hspace{1cm}& \text{for}~ i=s, \label{4LV5}\\ 
& & \notag \\
\text{and}  & x_{1}\frac{\partial f_{i}}{\partial x_{1}}+x_{3}\frac{\partial f_{i}}{\partial x_{3}}=-(i+1) (x_{1}+x_{3})f_{i+1} \hspace{1cm}& \text{for}~  0 \leq i \leq s-1   \label{4LV6}. 
\end{eqnarray}
Equation \eqref{4LV5} gives $f_{s} \in K$. Let $f_{s} = a_s$.\\
For $i=s-1$, from equation \eqref{4LV6}, we have
\begin{equation}
x_{1}\frac{\partial f_{s-1}}{\partial x_{1}}+x_{3}\frac{\partial f_{s-1}}{\partial x_{3}}=-s (x_{1}+x_{3})a_{s}.\label{4LV7}    
\end{equation}
Let $f_{s-1}=\sum_{j=0}^{p} g_{j}$, where $g_{j}$ is a homogeneous polynomial of degree $j$. We know for any form $g_{j}$ of total degree $j$
\begin{equation}
x_{1}\frac{\partial g_{j}}{\partial x_{1}}+x_{3}\frac{\partial g_{j}}{\partial x_{3}}=jg_{j}.\label{4LV8}   
\end{equation}
Thus, from equations \eqref{4LV7} and \eqref{4LV8}, we have $\sum_{j=1}^{p} jg_{j}=-s(x_{1}+x_{3})a_{s}$. This implies $j=1$ and $g_{1}=-s(x_{1}+x_{3})f_{s}$. So, $f_{s-1}=-s(x_{1}+x_{3})a_{s}+ a_{s-1}$, where $a_{s-1}(=g_{0}),a_{s} \in K$. \par
For $i=s-2$, from equation \eqref{4LV6}, we have
\begin{align}
x_{1}\frac{\partial f_{s-2}}{\partial x_{1}}+x_{3}\frac{\partial f_{s-2}}{\partial x_{3}}&=-(s-1) (x_{1}+x_{3})f_{s-1}\notag  \\ &=a_{s} s(s-1)(x_{1}+x_{3})^{2}-a_{s-1}(s-1)(x_{1}+x_{3})    .\label{4LV9}    
\end{align} 
Again, by using the argument as above, we get $f_{s-2}=a_{s}\frac{s(s-1)}{2}(x_{1}+x_{3})^{2}-a_{s-1}(s-1)(x_{1}+x_{3})+a_{s-2}$, where $a_{s-2} \in K$. Indeed, by induction, for $0\leq i \leq s$, we have 
\begin{align}
f_{s-i}= (-1)^{i} \binom{s}{i}a_{s}(x_{1}+x_{3})^{i}+ &(-1)^{i-1} \binom{s-1}{i-1}a_{s-1}(x_{1}+x_{3})^{i-1}+\cdots \notag \\& + (-1) \binom{s-(i-1)}{1}a_{s-i+1}(x_{1}+x_{3})+a_{s-i}, \label{4LV10}
\end{align}
where $a_{s-i} \in K$. Thus $\overline{c_{20}} \in K[x_{2},x_{1}+x_{3}]$.
\newline 
\textbf{Claim 2:} ~ $c_{2j} \in K[x_{2},x_{1}+x_{3}]$ for $j \geq 1$.\newline
By comparing the coefficient of $x_{4}^{n_{2}+1}$ in equation \eqref{4LV}, we have
\begin{equation}
x_{1}\frac{\partial c_{2n_{2}} }{\partial x_{1}}+x_{3}\frac{\partial c_{2n_{2}} }{\partial x_{3}}=0.
\end{equation}
The above equation gives $c_{2n_{2}} \in K[x_{2}]$. \\
Now, by comparing the coefficient of $x_{4}^{n_{2}}$ in equation \eqref{4LV}, we have
\begin{align}
x_{2}(x_{1}+x_{3})\frac{\partial c_{2n_{2}}}{\partial x_{2}}+ x_{1}\frac{\partial c_{2(n_{2}-1)} }{\partial x_{1}}+x_{3}\frac{\partial c_{2(n_{2}-1)} }{\partial x_{3}}+ n_{2}(x_{1}+x_{3})c_{2n_{2}}=c_{2n_{2}}(x_{1}+x_{3}). \notag   \end{align}
The above equation can be rewritten as 
\begin{equation}
  x_{1}\frac{\partial c_{2(n_{2}-1)} }{\partial x_{1}}+x_{3}\frac{\partial c_{2(n_{2}-1)} }{\partial x_{3}}=-(x_{1}+x_{3})(x_{2}\frac{\partial c_{2n_{2}}}{\partial x_{2}}+(n_{2}-1)c_{2n_{2}}).\label{4LV11}    
\end{equation}
Observe that $x_{2}\frac{\partial c_{2n_{2}}}{\partial x_{2}}+(n_{2}-1)c_{2n_{2}} \in K[x_{2}]$. Let $c_{2(n_{2}-1)}=\sum_{j=0}^{p} h_{j(n_{2}-1)}$, where $h_{j(n_{2}-1)} \in K[x_{2}][x_{1},x_{3}]$, is a homogeneous polynomial of degree $j$ with coefficients from $K[x_{2}]$. 
Then we have $$\sum_{j=1}^{p} jh_{j(n_{2}-1)}=-(x_{2}\frac{\partial c_{2n_{2}}}{\partial x_{2}}+(n_{2}-1)c_{2n_{2}})(x_{1}+x_{3}).$$ This implies $j=1$ and 
 $c_{2(n_{2}-1)}= \overline{h}_{0(n_{2}-1)}+\overline{h}_{1(n_{2}-1)} (x_{1}+x_{3}) $ $\in K[x_{2},x_{1}+x_{3}]$, where $\overline{h}_{0(n_{2}-1)}=h_{0(n_{2}-1)}$ and $\overline{h}_{1(n_{2}-1)}=-(x_{2}\frac{\partial c_{2n_{2}}}{\partial x_{2}}+(n_{2}-1)c_{2n_{2}})$. \newline
Now, by induction, one can show that whenever $c_{2i} \in K[x_{2},x_{1}+x_{3}]$, then $c_{2(i-1)} \in K[x_{2},x_{1}+x_{3}]$. Hence,  $c_{2j} \in K[x_{2},x_{1}+x_{3}]$ for all $j$.
\end{proof}

From Proposition \ref{propinf}, $\rho(x_{4})=x_{2}+x_{4}-\rho(x_{2}) \in K[x_{1}+x_{3},x_{2}, x_{4}]$. Using a similar argument 
one can also observe that $\rho(x_{1}),~ \rho(x_{3}) \in K[x_{1},x_{3},x_{2}+x_{4}]$. Further, by symmetry, it is easy to observe that if $\rho(x_{1}+x_{3})=x_{2}+x_{4}$ and $\rho(x_{2}+x_{4})=x_{1}+x_{3}$, then $\rho(x_{1}),~\rho(x_{3}) \in K[x_{1}+x_{3},x_{2},x_{4}]$ and $\rho(x_{2}),~\rho(x_{4}) \in K[x_{1},x_{3},x_{2}+x_{4}]$.  
Now, we show that if $s_1$ holds, then the isotropy group of $d$ contains non-linear automorphisms. 

\begin{theorem}\label{infinite}
 Let $d$ be the derivation as defined in equation \eqref{LVD4}. If $s_{1}$ holds, then the isotropy group of $d$ is not finite.   
\end{theorem}
 \begin{proof}
Consider linear automorphism $\sigma$ and non-linear automorphism $\rho$ given by:
$$\begin{aligned}
& \sigma(x_{1})=x_{1}, &&   \rho(x_{1})=x_{1},\\
&\sigma(x_{2})=c_{2}x_{2}+c_{4}x_{4}, && \rho(x_{2})=\alpha x_{2}+(x_{2}-(x_{1}+x_{3}))x_{4}+x_{4}^{2},\\
&\sigma(x_{3})=x_{3}, && \rho(x_{3})=x_{3},\\
&~\rho(x_{4})=(1-c_{2})x_{2}+(1-c_{4})x_{4}, && \rho(x_{4})=(1-\alpha)x_{2}+(1+x_{1}+x_{3}-x_{2})x_{4}-x_{4}^{2}.
\end{aligned}$$

where $\alpha, c_i \in K$. Then, $\sigma, ~ \rho \in \text{Aut}(K[X])_{d}$. Since $K$ is infinite, the isotropy group of $d$ contains infinitely many elements. 
\end{proof}

\section{\textbf{ Isotropy group of Lotka-Volterra derivation of \texorpdfstring{$K[x_{1},\ldots,x_{n}]$}{} }}

In this section, we study the isotropy group of Lotka-Volterra derivations of $K[x_1, \ldots, x_n]$, where $n \geq 5$. 
Let
\begin{equation}\label{5plusV}
d=x_{1}(x_{n}-C_{1}x_{2})\frac{\partial }{\partial x_{1}}+x_{2}(x_{1}-C_{2}x_{3})\frac{\partial }{\partial x_{2}}+\cdots+x_{n}(x_{n-1}-C_{n}x_{1})\frac{\partial }{\partial x_{n}},
\end{equation}
where $C_i \in K$ for all $ i \geq 1$. It is shown that the isotropy group of $d$ is finite for all values of $C_{i}$'s.  
First, we consider the case when $C_{i}$'s are non-zero scalars. 

\begin{theorem}\label{maintheorem}
Let $d$ be the Lotka-Volterra derivation as defined in equation \eqref{5plusV}. If $C_{i} \in K^{*}$, then the isotropy group of $d$ is finite.
\end{theorem}
  \begin{proof}
Let $\rho \in \text{Aut}(K[X])_{d}$. By Corollary \ref{reducedtotaldegree}, we know that $\rho(x_{i})$'s are polynomials in $K[X]$ of total degree $1$. Let 
$$\rho(x_{i})=c_{i}+\sum_{j=1}^n c_{ij}x_{j},$$
where $c_{i}, c_{ij} \in K$, for all $ i, j \in \mathbb{Z}_{n}$.\newline
For $i \in \mathbb{Z}_{n}$, the relation $d\rho(x_{i})=\rho d(x_{i})$ gives
\begin{equation}\label{nvariable1}
\begin{aligned}
\sum_{j=1}^n c_{ij}x_{j}(x_{j-1}-C_{j}x_{j+1})=
(c_{i}+ \sum_{j=1}^n c_{ij}x_{j}) 
&\big((c_{i-1}+\sum_{j=1}^n c_{(i-1)j}x_{j})-C_{i}(c_{i+1}+\sum_{j=1}^n c_{(i+1)j}x_{j})\big).
\end{aligned}
\end{equation}
By comparing the coefficient of $x_{j}^{2}$  in equation \eqref{nvariable1}, we have 
\begin{eqnarray}
    c_{ij}(c_{(i-1)j}-C_{i}c_{(i+1)j})=0 \hspace{1cm} \forall ~~i, j \in \mathbb{Z}_{n}. \label{nvariable2}
\end{eqnarray}
Further, by comparing the coefficient of $x_{q}x_{r}$, where $1 \leq q,r \leq n$, and $q \neq \{r,r\pm 1\} $ in equation \eqref{nvariable1}, we have 
\begin{eqnarray}
    c_{iq}(c_{(i-1)r}-C_{i}c_{(i+1)r})+c_{ir}(c_{(i-1)q}-C_{i}c_{(i+1)q})=0 \hspace{1cm} \forall ~~ i. \label{nvariable3}
\end{eqnarray}
Also, we know that $\rho(x_{i-1}-C_{i}x_{i+1}) \notin K$ for all $ i \in \mathbb{Z}_{n}$. 
In particular, for $i =k$, 
$\rho(x_{k-1}-C_{k}x_{k+1}) 
=(c_{k-1}-C_{k}c_{k+1})+\sum_{j=1}^{n} (c_{(k-1)j}-C_{k}c_{(k+1)j})x_{j} \notin K$. Hence, there exists a $t \in \mathbb{Z}_{n}$ such that $c_{(k-1)t}-C_{k}c_{(k+1)t} \neq 0$, which implies, $c_{kt}=0$. \\
Now, from equation \eqref{nvariable3} for $i=k,~r=t $ and $q \neq \{t,t\pm 1 \}$, we have 
\begin{eqnarray}
    c_{kq}(c_{(k-1)t}-C_{k}c_{(k+1)t})=0 ~~\Longrightarrow ~c_{kq}=0 ~~~~\text{and}  \notag
\end{eqnarray}
 $$ \rho(x_{k})=c_{k}+c_{k(t-1)}x_{t-1}+c_{k(t+1)}x_{t+1}\in K[x_{t-1},x_{t+1}].$$  \par
Now, by using the coefficient of $x_{t}^{2}$ in equation \eqref{nvariable2}, for  $i=k+1$, we have $c_{(k+1)t}(c_{kt}-C_{k+1}c_{(k+2)t})=0$. This implies $c_{(k+1)t}c_{(k+2)t}=0$. Recursively, from equation \eqref{nvariable2}, we get $c_{it}c_{(i+1)t}=0$ for all $i \in \mathbb{Z}_{n}$. \par
Since $c_{(k-1)t}-C_{k}c_{(k+1)t} \neq 0$, at least one of the $c_{(k-1)t}$ or $c_{(k+1)t}$ is not equal to zero. Without loss of generality, we assume $c_{(k+1)t} \neq 0$. Hence, $c_{(k+2)t}=0$, and from equation \eqref{nvariable3} for $i=k+1,~r=t $ and $q \neq \{t,t\pm 1 \}$, we have
$$c_{(k+1)t}(c_{kq}-C_{k+1}c_{(k+2)q})=0.$$
The above equation gives $c_{(k+2)q}=0$ for all $q \neq \{t,t\pm 1\}$. So, $$\rho(x_{k+2})=c_{k+2}+c_{(k+2)(t-1)}x_{t-1}+c_{(k+2)(t+1)}x_{t+1}\in K[x_{t-1},x_{t+1}].$$ 

\smallskip
\textbf{Claim:} $c_{it}=0$ for all  $i \in \mathbb{Z}_{n}$ and $i \neq k+1$.

We have observed that  $c_{kt}=c_{(k+2)t}=0$. We first show $c_{(k+3)t}=0$. Suppose not, then  $c_{(k+3)t}c_{(k+4)t}=0$, gives $c_{(k+4)t}=0$ and from equation \eqref{nvariable3} for $i=k+3,~r=t $ and $q \neq \{t,t\pm 1 \}$, we get 
$$c_{(k+3)t}(c_{(k+2)q}-C_{k+3}c_{(k+4)q})=0.$$ 
The above equation gives $c_{(k+4)q}=0$ for all $q \neq \{t,t\pm 1\}$, and $$\rho(x_{k+4})=c_{k+4}+c_{(k+4)(t-1)}x_{t-1}+c_{(k+4)(t+1)}x_{t+1}\in K[x_{t-1},x_{t+1}].$$ 
This is not possible as $\rho$ is an automorphism and $\rho(x_{k}),\rho(x_{k+2}) \in K[x_{t-1},x_{t+1}]$. 
Now, we show $c_{(k+4)t} = 0$. If not, then from equation \eqref{nvariable3} for $i=k+3,~r=t $ and $q \neq \{t,t\pm 1 \}$, we have
$$c_{(k+3)q}(c_{(k+2)t}-C_{k+3}c_{(k+4)t})=0 \Longrightarrow c_{(k+3)q}=0 .$$ 
 
Hence, $\rho(x_{k+3})=c_{k+3}+c_{(k+3)(t-1)}x_{t-1}+c_{(k+3)(t+1)}x_{t+1}\in K[x_{t-1},x_{t+1}]$, a contradiction. \par
Proceeding with a similar argument as above, one can show that  $c_{it}=0$ for  $i \neq k+1 $.  \par

\smallskip

Since $\rho(x_{k}-C_{k+1}x_{k+2})=(c_{k}-C_{k+1}c_{k+2})+(c_{k(t-1)}-C_{k+1}c_{(k+2)(t-1)})x_{t-1}+(c_{k(t+1)}-C_{k+1}c_{(k+2)(t+1)})x_{t+1} \notin K$, either $c_{k(t-1)}-C_{k+1}c_{(k+2)(t-1)} \neq 0$ or $c_{k(t+1)}-C_{k+1}c_{(k+2)(t+1)} \neq 0$. \par

\textbf{Case 1.} $c_{k(t+1)}-C_{k+1}c_{(k+2)(t+1)} \neq 0$. By using the coefficient of $x_{t+1}^{2}$ in equation \eqref{nvariable2}, for $i=k+1$, we have $c_{(k+1)(t+1)}=0$. Now, from equation \eqref{nvariable3} for $i=k+1,~r=t+1 $ and $q \neq \{t,t+1,t+2 \}$, we get 
$c_{(k+1)q}=0$ and $$\rho(x_{k+1})=c_{k+1}+c_{(k+1)t}x_{t}+c_{(k+1)(t+2)}x_{t+2}\in K[x_{t},x_{t+2}].$$ 

\textbf{Subcase 1.} $c_{(k+2)(t+1)} \neq 0$. In this case, $c_{(k+2)(t+1)}(c_{(k+1)(t+1)}-C_{k+2}c_{(k+3)(t+1)})= 0$. This implies $c_{(k+3)(t+1)}= 0$. \\
Recursively, we have $c_{i(t+1)}c_{(i+1)(t+1)}= 0$ for all $ i \in \mathbb{Z}_{n}$. 

Further, observe that $c_{(k+1)(t+1)}-C_{k+2}c_{(k+3)(t+1)}= 0$, together with equation \eqref{nvariable3} for $i=k+2,~r=t+1 $ and $q \neq \{t,t+1,t+2 \}$, gives
$$c_{(k+2)(t+1)}(c_{(k+1)q}-C_{k+2}c_{(k+3)q})=0.$$ 
Hence, $c_{(k+3)q}=0$ for all $q \neq \{t,t+1,t+2\}$ as $c_{(k+1)q}=0$ for all $q \neq \{t,t+1,t+2\}$ and  $$\rho(x_{k+3})=c_{k+3}+c_{(k+3)t}x_{t}+c_{(k+3)(t+2)}x_{t+2}\in K[x_{t},x_{t+2}].$$ 
Now, one can show that $c_{i(t+1)}=0$ for all $i \in \mathbb{Z}_{n}$ except $i = k+2$
by using an argument similar to the case of $c_{it}=0$ for  $i \neq k+1$. In particular, for $i=k$, $c_{k(t+1)}=0$ and  $\rho(x_{k})$ is given by $$\rho(x_{k})=c_{k}+c_{k(t-1)}x_{t-1}.$$

Note that the way we have shown $c_{it}=0$ for all $i$ except $i = k+1$ and $c_{i(t+1)}=0$ for all $i$ except $i = k+2$, we also get $c_{(k+2)q}=c_{(k+4)q}=0$ for all $q$ except $q = \{t+1,t+2,t+3\}$, and
$c_{i(t+2)}=0$ for all $i$ except $i = k+3$. Now,
\begin{align}
 & \rho(x_{k+1})=c_{k+1}+c_{(k+1)t}x_{t}, \notag \\
 &\rho(x_{k+2})=c_{k+2}+c_{(k+2)(t+1)}x_{t+1},  \notag \\
 &\rho(x_{k+3})=c_{k+3}+c_{(k+3)(t+2)}x_{t+2},\notag \\
 &\rho(x_{k+4})=c_{k+4}+c_{(k+4)(t+3)}x_{t+3}.\notag
\end{align}

By continuing the above argument, we get $\rho$ of the form $$\rho(x_{k+j})=c_{k+j}+c_{(k+j)(t+j-1)}x_{t+j-1} \hspace{1cm} \forall~~ j \in \mathbb{Z}_{n}.$$\par
\textbf{Subcase 2.} $c_{k(t+1)} \neq 0$. Using a similar argument as in \textbf{Subcase 1}, we can show that $\rho$ will be of the form 
$$\rho(x_{k+j})=c_{k+j}+c_{(k+j)(t-j+1)}x_{t-j+1} 
\hspace{1cm} \forall ~~j \in \mathbb{Z}_{n}.$$\par
\textbf{Case 2.} $c_{k(t-1)}-C_{k+1}c_{(k+2)(t-1)} \neq 0$. Using a similar argument as in \textbf{Case 1},\\ If $c_{k(t-1)}\neq 0$, then
$$\rho(x_{k+j})=c_{k+j}+c_{(k+j)(t+j-1)}x_{t+j-1} \hspace{1cm} \forall~~ j \in \mathbb{Z}_{n}.$$
If $c_{(k+2)(t-1)} \neq 0$, then $$\rho(x_{k+j})=c_{k+j}+c_{(k+j)(t-j+1)}x_{t-j+1} 
\hspace{1cm} \forall ~~j \in \mathbb{Z}_{n}.$$ 

\smallskip
Now, we show that $c_{ij}$ can take only finitely many values and $c_{i}=0$ for all $i$.\newline \par
If $\rho$ is of the form $\rho(x_{k+j})=c_{k+j}+c_{(k+j)(t+j-1)}x_{t+j-1}$, then $d \rho(x_{k+j})= \rho d (x_{k+j})$ gives 
\begin{align}
c_{(k+j)(t+j-1)}x_{t+j-1}(x_{t+j-2}-&C_{t+j-1}x_{t+j})=
   (c_{k+j}+c_{(k+j)(t+j-1)}x_{t+j-1})\big((c_{k+j-1}\notag \\& + c_{(k+j-1)(t+j-2)}x_{t+j-2})-C_{k+j}(c_{k+j+1}+c_{(k+j+1)(t+j)}x_{t+j})\big). \label{case1}    
\end{align}

By comparing the coefficients of $x_{t+j-2}$, $x_{t+j-1}x_{t+j-2}$ and $x_{t+j-1}x_{t+j}$ in equation \eqref{case1}, we have $c_{k+j}=0$, $c_{(k+j-1)(t+j-2)}=1$, and $-C_{t+j-1}=-C_{k+j}c_{(k+j+1)(t+j)}$, respectively.
Since it is true for all $ j \in \mathbb{Z}_{n}$, the automorphism $\rho$ is given by $\rho(x_{k+j})=x_{t+j-1}$ with $C_{k+j}=C_{t+j-1}$. \newline

Else, $\rho$ will be of the form $\rho(x_{k+j})=c_{k+j}+c_{(k+j)(t-j+1)}x_{t-j+1}$
and  $d \rho(x_{k+j})= \rho d (x_{k+j})$ gives 
\begin{align}
    c_{(k+j)(t-j+1)}x_{t-j+1}(x_{t-j}-&C_{t-j+1}x_{t-j+2})=(c_{k+j}+c_{(k+j)(t-j+1)}x_{t-j+1}) \big((c_{k+j-1}\notag \\& + c_{(k+j-1)(t-j+2)}x_{t-j+2})-C_{k+j}(c_{k+j+1}+c_{(k+j+1)(t-j)}x_{t-j})\big).\label{case2}
\end{align} 
By comparing the coefficients in equation \eqref{case2}, we have $c_{k+j}=0$, $-C_{k+j}c_{(k+j+1)(t-j)}=1$, and $c_{(k+j-1)(t-j+2)}=-C_{t-j+1}$. Since it is true for all $ j \in \mathbb{Z}_{n}$, the automorphism $\rho$ is given by $\rho(x_{k+j})=-C_{t-j}x_{t-j+1}$ with $C_{k+j-1}C_{t-j}=1$.

In both cases, $\rho$ can take only finitely many values. Hence, the isotropy group of $d$ is finite.
  \end{proof}

For $k,t \in \mathbb{Z}_{n}$, let $S_{1}$ and $S_{2}$ be
\begin{align}
 & S_{1}= \left \{ \rho ~|~ \rho(x_{k+j})=x_{t+j-1}~ \text{with}~C_{k+j}=C_{t+j-1}~~ \forall ~~j \in \mathbb{Z}_{n} \right \}, \notag\\
 &S_{2}=\left \{ \rho ~| ~\rho(x_{k+j})=-C_{t-j}x_{t-j+1}~ \text{with}~C_{k+j-1}C_{t-j}=1~\forall ~~j \in \mathbb{Z}_{n} \right \}. \notag  
\end{align}

\begin{remark} \label{mainremark}
\begin{enumerate}
\item From Theorem \ref{maintheorem},  
$\left | S_{1} \right | \leq n,~\left | S_{2} \right | \leq n$, and $\text{Aut}(K[X])_{d}=S_{1}\cup S_{2}$.
\item If $\rho \in S_{1}$, and $t=k+1$, then $\rho$ is an identity automorphism; in this case, there are no conditions on $C_{i}$.
\item If $\rho \in S_{1}$, then we can write $\rho$ in the cyclic notation as $(x_{k+j}~~~x_{t+j-1}~~~\ldots ~~~ x_{mt+j-m-(m-1)k} ~~~\ldots)$.
Observe that if there is a cycle of length $m$, then every cycle in $\rho$ is of length $m$. So, if $\rho$ is of order $p$, then $\rho$ is a product of $n/p$ different cycles of length $p$.
\item If $\rho \in S_{2}$, then $\rho ^{2}=id$. 
\end{enumerate}
   \end{remark}

Based on Theorem \ref{maintheorem}, we can now explicitly compute the isotropy group of Lotka-Volterra derivations for particular values of $C_i$'s.  

\begin{corollary} \label{dihedralcorollary}
Let $d$ be the derivation as defined in equation \eqref{5plusV}. If $C_{1}=\ldots=C_{n}$ and $C_{1} \in \{\pm 1\}$, then the isotropy group of $d$ is isomorphic to the dihedral group $D_{2n}$ of order $2n$.
\end{corollary}
\begin{proof}
From Theorem \ref{maintheorem}, for $t=k$, the automorphisms $\rho(x_{k+j})=x_{k+j-1}$ and $\sigma(x_{k+j})= C_{1}x_{k-j+1}$ for all $j \in \mathbb{Z}_{n}$, are in the isotropy group of $d$. Also, one can easily observe that    $\rho^{n}=\sigma^{2}=id$ and $\rho \sigma = \sigma \rho ^{-1}$. Further, if we have an automorphism $\phi \in \text{Aut}(K[X])_{d}$, 
then $\phi= \rho^{m} \sigma$ for a suitable $m$. Hence,
\begin{center}
    $\text{Aut}(K[X])_{d}=\{<\rho, \sigma>~|~\rho^{n}=\sigma^{2}=id~ \text{and}~ \rho \sigma = \sigma \rho ^{-1} \}$.
\end{center}
\end{proof}

\begin{corollary} \label{dihedral}
Let $d$ be the derivation as defined in equation \eqref{5plusV} and $n$ be an even integer. If $C_{2i-1}=C_{2i+1}=\pm1$ and $C_{2i}=C_{2i+2}=\mp1$ for all $i \in \mathbb{Z}_{n}$, respectively, then the isotropy group of $d$ is isomorphic to the dihedral group $D_{n}$ of order $n$.
\end{corollary}
\begin{proof}
 Let $n=2m$.
From Theorem \ref{maintheorem}, for $t=k-1$, the automorphisms 
$\rho(x_{k+j})=x_{k+j-2} $ and $\sigma(x_{k+j})=-C_{k-j-1}x_{k-j} $ for $j \in \mathbb{Z}_{n}$, are in the isotropy group of $d$. Further, $\rho ^{m}=\sigma^{2}=id$, and $ \rho \sigma = \sigma \rho ^{-1}$. Thus, 
\begin{center}
    $D_{n}=\{<\rho, \sigma>~|~\rho^{m}=\sigma^{2}=id~ \text{and}~ \rho \sigma = \sigma \rho ^{-1} \} \subseteq \text{Aut}(K[X])_{d}$.
\end{center}

Further, it is easy to observe that no other element from $S_1$ or $S_2$ is in the isotropy group of $d$.
Hence, the isotropy group of $d$ is isomorphic to $D_{n}$.
\end{proof}

\begin{corollary}
Let $d$ be the derivation as defined in equation \eqref{5plusV}. If, for $1 \leq i,j \leq n$,  $C_{i}=C_{j}$ and $C_{i}C_{j} \neq 1$, then the isotropy group of $d$ is isomorphic to $\mathbb{Z}_{n}$.
\end{corollary}
    \begin{proof}
The result follows directly from Theorem \ref{maintheorem}.
\end{proof}

\noindent Now, we consider the case when at least one of the $C_{i}$'s is zero. We prove that if $C_{i}$'s are zero for all $i$, then the isotropy group of $d$ is isomorphic to a cyclic group of order $n$. Further, if $C_{i}$ is zero for at least one $i$, but not all the $C_{i}$'s are identically zero, then the isotropy group of $d$ is isomorphic to a subgroup of a cyclic group of order $n$.

\begin{theorem}\label{maintheorem2}
Let $d$ be the derivation as defined in equation \eqref{5plusV}. 
\begin{enumerate}
\item  If $C_{i}=0$ for all $i$, then the isotropy group of $d$ is isomorphic to $\mathbb{Z}_{n}$.
\item If $C_i =0$ for at least one $i$ but not for every $i$, then the isotropy group of $d$ is isomorphic to a subgroup of $\mathbb{Z}_{n}$.
\end{enumerate}
\end{theorem}
\begin{proof}
Let $\rho \in \text{Aut}(K[X])_{d}$. By Corollary \ref{reducedtotaldegree}, we know that $\rho(x_{i})$'s are polynomials in $K[X]$ of total degree $1$. Let 
$$\rho(x_{i})=c_{i}+\sum_{j=1}^n c_{ij}x_{j},$$
where $c_{i}, c_{ij} \in K$, for all $ i, j \in \mathbb{Z}_{n}$.\newline
 
Let $C_{p}=0$ for some $p \in \mathbb{Z}_{n} $. Then, $d\rho(x_{p})=\rho d(x_{p})$ gives
\begin{equation}\label{nvariable5}
\begin{aligned}
\sum_{j=1}^{p-1} c_{pj}x_{j}(x_{j-1}-&C_{j}x_{j+1})+c_{pp}x_{p}x_{p-1}+\sum_{j=p+1}^{n} c_{pj}x_{j}(x_{j-1}-C_{j}x_{j+1})\\
&=(c_{p}+ \sum_{j=1}^n c_{pj}x_{j}) 
(c_{p-1}+\sum_{j=1}^n c_{(p-1)j}x_{j}).
\end{aligned}
\end{equation}

Also, for $i \neq p$, $d\rho(x_{i})=\rho d(x_{i})$ gives

\begin{equation}\label{nvariable4}
\begin{aligned}
\sum_{j=1}^{p-1} c_{ij}x_{j}(x_{j-1}-&C_{j}x_{j+1})+c_{ip}x_{p}x_{p-1}+\sum_{j=p+1}^{n} c_{ij}x_{j}(x_{j-1}-C_{j}x_{j+1})\\
&=(c_{i}+ \sum_{j=1}^n c_{ij}x_{j}) 
\big((c_{i-1}+\sum_{j=1}^n c_{(i-1)j}x_{j})-C_{i}(c_{i+1}+\sum_{j=1}^n c_{(i+1)j}x_{j})\big).
\end{aligned}
\end{equation}

By comparing the coefficient of $x_{j}^{2}$ in equations \eqref{nvariable5} and \eqref{nvariable4}, we have 
\begin{eqnarray}
    c_{pj}c_{(p-1)j}=0, \label{nvariable6}
\end{eqnarray}
 and for $i \neq p$ 
\begin{eqnarray}
    c_{ij}(c_{(i-1)j}-C_{i}c_{(i+1)j})=0. \label{nvariable7}
\end{eqnarray} 
The equations \eqref{nvariable6} and \eqref{nvariable7}, together gives
\begin{eqnarray}
    c_{ij}c_{(i-1)j}=0 ~\hspace{1cm}~\forall~~i,j. \label{nvariable8}
\end{eqnarray}
Now, for $1 \leq q,r \leq n$ and $q \neq \{r,r\pm 1\}$, by comparing the coefficient of $x_{q}x_{r}$ in equation \eqref{nvariable5}, we have 
\begin{eqnarray}
    c_{pq}c_{(p-1)r}+c_{pr}c_{(p-1)q}=0,\label{nvariable9}
\end{eqnarray}
and by comparing the coefficient of $x_{q}x_{r}$ in equation \eqref{nvariable4}, we have 
\begin{eqnarray}
    c_{iq}(c_{(i-1)r}-C_{i}c_{(i+1)r})+c_{ir}(c_{(i-1)q}-C_{i}c_{(i+1)q})=0 \hspace{1cm} \text{(where}~ i\neq p).\label{nvariable10}
\end{eqnarray}
By putting $i=p-1$ in equation \eqref{nvariable10} and  substituting equation \eqref{nvariable9}, we have
\begin{align}
 & c_{(p-1)q}c_{(p-2)r}+c_{(p-1)r}c_{(p-2)q}-C_{p-1}(c_{pr}c_{(p-1)q}+c_{pq}c_{(p-1)r})=0. \notag\\
  \Longrightarrow ~~~ & c_{(p-1)q}c_{(p-2)r}+c_{(p-1)r}c_{(p-2)q}=0.\notag  
\end{align}
Recursively, we get 
\begin{eqnarray}
    c_{iq}c_{(i-1)r}+c_{ir}c_{(i-1)q}=0 \hspace{1cm} \forall~~ i. \label{nvariable11}
\end{eqnarray}
Now, using a similar argument as shown in Theorem \ref{maintheorem}, we can show that $\rho$ will be of the form
$$\rho(x_{k+j})=c_{k+j}+c_{(k+j)(t+j-1)}x_{t+j-1} \hspace{1cm} \forall~~ j \in \mathbb{Z}_{n}$$ 
or
$$\rho(x_{k+j})=c_{k+j}+c_{(k+j)(t-j+1)}x_{t-j+1} \hspace{1cm} \forall ~~j \in \mathbb{Z}_{n}.$$
Next, we show that $c_{ij}$ can take only finitely many values.\newline \par
\textbf{Case 1.} If $C_{i}=0$ for all $i$. In this case, $d(x_{i})=x_{i}x_{i-1}$.\newline \par
\textbf{Sub-Case 1.} If $\rho$ is of the form $\rho(x_{k+j})=c_{k+j}+c_{(k+j)(t+j-1)}x_{t+j-1}$.
Then $d \rho(x_{k+j})= \rho d (x_{k+j})$ gives 
\begin{eqnarray}
   c_{(k+j)(t+j-1)}x_{t+j-1}x_{t+j-2}=
   (c_{k+j}+c_{(k+j)(t+j-1)}x_{t+j-1})(c_{k+j-1}+c_{(k+j-1)(t+j-2)}x_{t+j-2}).\notag
\end{eqnarray}

By equating various coefficients in the above equation, we have $c_{k+j}=c_{k+j-1}=0$ and $c_{(k+j-2)(t+j-2)}=1 $. So, $\rho$ is of the form $\rho(x_{k+j})=x_{t+j-1}$ for all $j \in \mathbb{Z}_{n}$.\par
\textbf{Sub-Case 2.} If $\rho$ is of the form $\rho(x_{k+j})=c_{k+j}+c_{(k+j)(t-j+1)}x_{t-j+1}$.
Then, $d \rho(x_{k+j})= \rho d (x_{k+j})$ gives
\begin{eqnarray}
   c_{(k+j)(t-j+1)}x_{t-j+1}x_{t-j}=
  (c_{k+j}+c_{(k+j)(t-j+1)}x_{t-j+1})(c_{k+j-1}+c_{(k+j-1)(t-j+2}x_{t-j+2}).\notag
\end{eqnarray}
Clearly, the above equation does not hold, and in this case, we have no automorphism.\par
\smallskip
So, $\rho$ is of the form 
$\rho(x_{k+j})=x_{t+j-1}$ for all $j \in \mathbb{Z}_{n}$. 
In particular, for $t=k$, $\rho(x_{k+j})=x_{k+j-1}$ with $\rho ^{n}=id$. Also, automorphism for any other value of $t$ will be a power of $\rho$. Hence, the isotropy group of $d$ is generated by $\rho$ and is isomorphic to $\mathbb{Z}_{n}$. \newline \par
\textbf{Case 2.} All $C_{i}$'s are not equal to zero.\newline \par
\textbf{Sub-Case 1.} If $\rho$ is of the form $\rho(x_{k+j})=c_{k+j}+c_{(k+j)(t+j-1)}x_{t+j-1}$.
Then, using an argument similar to the proof of  Theorem \ref{maintheorem} (\textbf{Case 1}), we have $\rho(x_{k+j})=x_{t+j-1}$ with $C_{k+j}=C_{t+j-1}$ for all $j \in \mathbb{Z}_{n}$. As, $C_{p}=0$, it enforces either $C_{p}=C_{t+p-k-1}=\ldots=C_{m(t-1)-mk+p}=0$ or $t=k+1$. 
If $t=k+1$, then $\rho$ is the identity automorphism. Further,  the isotropy group is generated by $\rho$, and the order of $\rho$ is a divisor of $n$.  

\textbf{Sub-Case 2.} If $\rho$ is of the form $\rho(x_{k+j})=c_{k+j}+c_{(k+j)(t-j+1)}x_{t-j+1}$, then we observe no automorphism.
\end{proof}

\begin{corollary}
 Let $d$ be the derivation as defined in equation \ref{LVD}. If $n \geq 3$ and $C_{i}=1$ for all $i$, then the isotropy group of $d$ is isomorphic to the dihedral group of order $2n$.   
\end{corollary}
\begin{proof}
 The proof follows from Corollaries \ref{C_{i}=1}, \ref{4C_{i}=1}, and \ref{dihedralcorollary}.   
\end{proof}

\section*{Acknowledgements}
The first author would like to acknowledge the financial support by Council of Scientific and Industrial Research, India (09/086(1441)/2020-EMR-I).

\bibliographystyle{abbrv}
\bibliography{biblo}

\end{document}